\documentclass[preprint,12pt,authoryear]{elsarticle}

\biboptions{longnamesfirst,angle,semicolon}

\usepackage{graphicx}%
\usepackage{multirow}%
\usepackage{amsmath,amssymb,amsfonts}%
\usepackage{amsthm}%
\usepackage{tikz-cd}%
\usepackage{mathrsfs}%
\usepackage[title]{appendix}%
\usepackage{xcolor}%
\usepackage{textcomp}%
\usepackage{manyfoot}%
\usepackage{booktabs}%
\usepackage{algorithm}%
\usepackage{algorithmicx}%
\usepackage{algpseudocode}%
\usepackage{listings}%
\usepackage{hyperref}

\usepackage{makeidx}
\makeindex

\newtheorem{theorem}{Theorem}[section]
\newtheorem{proposition}[theorem]{Proposition}
\newtheorem{corollary}[theorem]{Corollary}
\newtheorem{lemma}[theorem]{Lemma}

\newdefinition{example}{Example}[section]
\newdefinition{remark}{Remark}[section]
\newdefinition{definition}{Definition}[section]

\newproof{pf}{Proof}

\usepackage{lineno}

\begin{document}

%\linenumbers
\begin{frontmatter}



\title{Topological incidence rings and the incidence functor} %% Article title


\author{João V. P. e Silva} %% Author name
\ead{joaovitorps@outlook.com}
%\affiliation{organization={},%Department and Organization

%% Abstract
\begin{abstract}
We focus on incidence rings, a class of (possibly infinite) matrix rings indexed by ordered sets. Some general properties about them are given, including how they are always the inverse limit of finite matrix rings, giving a natural way to define their topology. It is also shown that the map that sends ordered sets to its incidence ring is a functor from the category of locally finite, preordered sets to the category of (topological) rings, and this functor sends colimits to limits. We discuss consequences of these results to the class of groups of units of incidence rings, and finish by giving some conditions for these groups to be solvable or not.
\end{abstract}


\begin{highlights}
\item The study of a natural topology for incidence algebras and their groups of units using categorical structures.
\item A proof that for every associative ring with unit, there is a contravariant functor that maps preordered sets to (topological) algebras and that it maps colimits to limits.
\item A proof that if the ring is also commutative, then there is a contravariant functor that maps preordered sets to (topological) groups and that it maps colimits to limits.
\item Using the categorical properties, we give conditions for the group of units of incidence algebras to be solvable or not.
\end{highlights}

%% Keywords

\end{frontmatter}

\maketitle

\section{Introduction}\label{Introduction}

%Given a preordered set $\Lambda$ and a ring $P$ we define the rings $M_{\Lambda,P)$, $AM_{\Lambda,P)$ and the groups $Tri^*(\Lambda,P)$, $AM_{\Lambda,P)$ which, under the asumption $P$ is a topological ring, can all be seen as topological rings/groups. This is a generalization of a work done on \cite{simple_matrices} where uncountably many simple totally disconnected locally compact second countable groups are built by indexing infinite matrices with the integers and letting $P$ be a finite field.\\

The study of infinite matrices traces back to Poincaré in the discussion of the well known Hill's equation. In \cite{history_infinite_matrices}, Bernkopf gives an introduction to the history of the study of infinite matrices and their importance for operator theory. \\
  In the present paper we focus on incidence rings, a class of rings described using locally finite ordered sets, that is, the sets $[x,y]=\{z; x\leqslant z\leqslant y\}$ are finite for every $x,y$ in the ordered set. This class of rings was first defined in \cite{Matrices1970} restricted to partially ordered sets, and it was used to give a counter-example for a non-published conjecture. Further work on incidence rings can be found in \cite{Matrices1980,Matrices1984,Matrices1985,Matrices1990,Matrices1996,Matrices1999,Matrices2002,Matrices2010} with results about Morita duality, the isomorphism problem for incidence rings, and other general properties of such rings.\ 
  In this article we focus on the theory of topological incidence rings. We show connections between topological incidence rings and the category of locally finite preordered sets by constructing incidence functors over a ring $P$, a collection of contravariant functors from the category of locally finite preordered sets to the category of (topological) rings that map colimits to limits. It is also shown that the class of incidence rings over a ring $P$ is generated by a countable set of incidence rings under the operations of products, inverse limits and pullbacks.\\
  Many of the results related to the category of incidence rings can also be proved to their group of units, as the unit functor is right adjoint. We also give conditions for these groups to be topological groups in the subspace topology.\\
%   We also study the groups of units of incidence rings. A generalization of the construction of matrix groups and rings from \cite{simple_matrices} is given for ordered sets that are not $\mathbb{Z}$-like \cite[Definition 2.1]{simple_matrices}. We also generalize the construction of non-compactly generated, non-discrete, simple groups \cite[Theorem 5.2]{simple_matrices} for a bigger class of groups that are not necessarily totally disconnected, nor locally compact, nor second countable.\\
%  Groups similar to the ones studied here have been described previously. The groups denoted here as $\textup{aGL}_\mathbf{N}(\mathbb{F}_q)$, for $\mathbf{N}$ the set of naturals with usual order and $\mathbb{F}_q$ the finite field with $q$ elements, was first built at \cite{first_almost_upper_triang} as the group $\textup{GLB}$. Some first study on the representation of such groups was done at \cite{representation_glb}.\\
%  Study of the group $\textup{GLB}$ and similarly defined groups over more general rings are also described in \cite{Hou2017CommutatorSO,GUPTA20124279,GUPTA201585,10.2307/2699674}.\\
%  We also prove sufficient conditions for the group of units of an incidence ring to be quasi-discrete or solvable. It then follows that if $G$ contains a matrix group with such property as an open, compact subgroup, then $G$ is always an elementary topological group in the Wesolek sense \cite{elementary_first}.\\
  All rings in this article are assumed to be associative with identity. Rings are assumed to be non-trivial, unless explicitly stated. Ring homomorphisms send the multiplicative identity to the multiplicative identity. Topological rings and groups are always Hausdorff spaces.

\subsection{Structure of the article}

Section \ref{background} of this article will focus on giving the necessary background for the definitions and results of the article. In Subsection \ref{ordered_sets} some definitions for ordered sets are given. These will be central for the definition of incidence rings, their groups of units, and to prove results for such classes of objects.\\
  In Section \ref{Topological incidence rings} we give the definition of incidence rings. It is also shown that these rings are inverse limits of finite matrix rings, allowing us to give a topological structure to the incidence rings.\\
  Section \ref{Categorial properties} focuses on connecting some properties from the incidence rings with their partially ordered sets. The initial focus is describing a category for locally finite preordered sets so that there exists a functor from this category to the category of topological rings. It is also shown that under such properties the functor is contravariant and sends colimits to limits. A corollary of this result is that all matrix rings over $P$ indexed by partially ordered sets can be built from a countable collection of finite matrices over $P$ under the operations of direct products, inverse limits, and pullbacks.\\
  Section \ref{Incidence functor and the group of units} focuses on the groups of units of an incidence ring. We show that when composing the functors from Section \ref{Categorial properties} with the functor that sends a ring to its group of units gives us a functor from the category of preordered sets to the category of (topological) groups. We finish by showing some conditions for a group of units of an incidence algebra to be solvable or not.
%  Section \ref{questions} gives a list of questions and alternative constructions that might also lead to interesting results.

\section{Background}\label{background}

\subsection{Ordered sets}\label{ordered_sets}

  For the context of our rings and groups, a definition that proves necessary is the following:

\begin{definition}
Let $(\Lambda,  \preceq)$ be a proset (preordered set). Given $s_1  \preceq s_2 \in \Lambda$ we define $[s_1, s_2]:=\{s\in \Lambda: s_1  \preceq s  \preceq s_2\}$. We will call these subsets of $\Lambda$ \textbf{intervals}. We say that the proset is \textbf{locally finite} if for every $s_1  \preceq s_2\in S$ the set $[s_1, s_2]$ is finite.
    %%\item The proset is locally infinite if for every $s\in S$ the subset $\mathcal{N}_1(s)$ has infinitely many elements.\index{locally infinite}[DO I USE THIS?]
    %%\item The proset has arbitrary large intervals if for every $s\prec t\in S$ there exists $s',t'\in S$ such that $[s,t] \varsubsetneq [s',t']$ .\index{arbitrary large intervals}[DO I USE THIS?]
    %%\item The proset is locally finite if for every $s\in S$ the set $\mathcal{N}_1(s)$ is finite.\index{locally finite}[DO I USE THIS?]
\end{definition}
  This locally finite type of property allows us to define multiplication and addition on these rings in a similar way to the finite-dimensional matrix rings. It will also be essential to build a nice topology for such groups/rings.\\
  As the central object of this text is algebraic structures defined in relation to an ordered set, defining some subsets and relations in the ordered sets will be essential for simplifying the proofs. This subsection will focus on giving these basic definitions and providing some results about the structure of ordered sets and their subsets.\\
\textbf{Notation}:
Let $\Lambda$ be a proset and $s_1\preceq s_2\in \Lambda$. We denote:
\begin{itemize}
    \item $s_1\precnsim s_2$ if $s_1 \preceq s_2$ and $s_1\nsim s_2$.
    \item $s_1\precnapprox s_2$ if $s_1 \preceq s_2$ and $s_1\neq s_2$.
    \item $s_1\npreceq s_2$ if it is not the case $s_1\preceq s_2$.
\end{itemize}
  Notice that for $\Lambda$ a proset the set $[s, s]$ is the set of all elements equivalent to $s$ under the preorder. Hence it is not always the case $[s, s]=\{s\}$. Notice that if for every $s\in \Lambda$, $[s, s]=\{s\}$ then $\Lambda$ is a poset (partially ordered set).\\

\begin{definition}\label{defi:general poset one} 
Given a proset $\Lambda$, we define the following:
\begin{itemize}
    \item Given $s\in \Lambda$, we define the \textbf{neighbourhoods} of $s$ as:
    $$\mathcal{N}_{n}(s):=\left\{\begin{array}{ll}
       [s, s]  & \mbox{for }n=0 \\
       \{t\in \Lambda: t  \preceq s \mbox{ or }s  \preceq t\} & \mbox{for }n=1 \\
    \bigcup_{t\in\mathcal{N}_{n-1}(s)}\mathcal{N}_1(t) & \mbox{for } n>1
    \end{array}\right.$$
    The \textbf{full collection of neighbours} $s$ is defined as $\mathcal{N}_\omega(s)=\bigcup_{n\in\mathbb{N}}\mathcal{N}_n(s)$.
    \item We say $s_1, s_2\in \Lambda$ are \textbf{independent} if $s_1\notin \mathcal{N}_{\omega}(s_2)$.
    \item A subset $\Lambda'\subset \Lambda$ is said to be \textbf{closed under intervals} if for every $s_1  \preceq s_2\in 
    \Lambda'$ then $[s_1, s_2]_{  \preceq}\subset \Lambda'$.
    \item Given $\Lambda'\subset \Lambda$ a subset, we say $\Lambda'$ is \textbf{convex} if it is closed under intervals and for every $s_1, s_2\in \Lambda$ there are $t_1, t_2, \ldots, t_n\in \Lambda'$ such that $t_1\in \mathcal{N}_1(s_1)$, $t_{i+1}\in \mathcal{N}_1(t_i)$, $s_2\in \mathcal{N}_1(t_n)$ for $1\leqslant i\leqslant n-1$. In other words, there is a path of intervals inside $\Lambda'$ connecting $s_1$ and $s_2$.
    \item Given $\{\Lambda_i\}_{i\in I}$ a collection such that for all $i\in I$, $\Lambda_i\subset \Lambda$ is convex in $\Lambda$ and for $i\neq j\in I$ then $\Lambda_i\bigcap \Lambda_j=\emptyset$, then $\{\Lambda_i\}_{i\in I}$ is called a \textbf{locally convex collection} of $\Lambda$.
\end{itemize}
\end{definition}

%Notice that for $s$ an element of $\Lambda$ it is always the case $\mathcal{N}_{\omega}(s)$ is the maximal convex set of $\Lambda$ containing $s$. It then follows that the relation of being independent is symmetric. 

\begin{definition}
Let $\Lambda$ be a proset and $s\in\Lambda$. We say that:
\begin{itemize}
    \item $s$ is a \textbf{maximal element} if given $t\in\Lambda$ is such that $s  \preceq t$ then $s=t$.
    \item $s$ is a \textbf{minimal element} if given $t\in\Lambda$ is such that $t  \preceq s$ then $s=t$.
    \item $s$ is the \textbf{maximum element} if for all $t\in \Lambda$ we have $t  \preceq s$.
    \item $s$ is the \textbf{minimum element} if for all $t\in \Lambda$ we have $s  \preceq t$.
\end{itemize} 
\end{definition}

The following examples illustrate the definition of intervals, neighbourhoods, and convex subsets as denoted above.

\begin{example}\label{ex:basic posets}\label{ex:omega complexity degree}
\begin{itemize}
\item $\mathbf{Q}$: The rational numbers with usual order is a partially ordered sets that is not locally finite.
\item $\mathbf{Zig}$: This is the poset with elements in the integers and order given by $2k>2k+1$ and $2k> 2k-1$, for $k$ an integer.\\
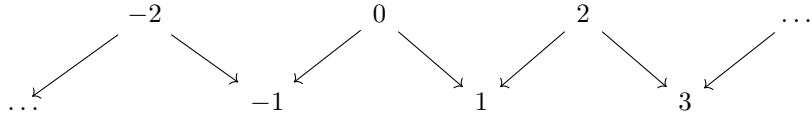
$$\begin{tikzcd}
       & -2 \arrow[ld] \arrow[rd] &    & 0 \arrow[ld] \arrow[rd] &   & 2 \arrow[ld] \arrow[rd] &   & \ldots \arrow[ld] \\
\ldots &                          & -1 &                         & 1 &                         & 3 &                  
\end{tikzcd}$$
The only intervals for $\textbf{Zig}$ are of the form $[n, n]=\{n\}$, $[2n, 2n-1]=\{2n, 2n-1\}$ and $[2n, 2n+1]=\{2n, 2n+1\}$ for some $n\in\mathbb{Z}$. In this case, $\mathcal{N}_0(0)=\{0\}$, $\mathcal{N}_n(0)=\{i\}_{-n\leqslant i\leqslant n}$ and, because no two elements are independent, $\mathcal{N}_\omega(0)=\mathbf{Zig}$. A subset $S\subset \mathbf{Zig}$ is convex if, and only if, there are $n_1, n_2\in\mathbf{Z}\bigcup\{\infty, -\infty\}$ such that $S=\{i\}_{n_1< i < n_2}$. This poset has no independent element. Every element is either maximal or minimal, but it doesn't have any maximum or minimum element.

\item $\mathbf{N}^*_d$: The non-zero natural numbers with order given by $a\leqslant b$ if $b=ak$ for some $k$ natural number.\\
Order on the first $6$ natural numbers.
$$
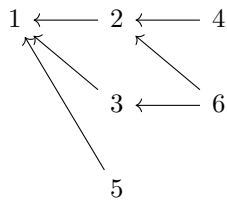
\begin{tikzcd}
1 & 2 \arrow[l]   & 4 \arrow[l]            \\
  & 3 \arrow[lu]  & 6 \arrow[l] \arrow[lu] \\
  & 5 \arrow[luu] &                       
\end{tikzcd}$$
The intervals of this poset can be described using arithmetic properties. For example $[\leqslant 30]=\{1, 2, 3, 5, 6, 10, 15, 30\}_{\mathbf{N}_d^*}$ (the divisors of $30$) and $[30\leqslant]_{\mathbf{N}_d^*}=\{30n\}_{n\in\mathbf{N}_d^*}$ (multiples of $30$). The interval $[3, 30]=\{3, 6, 15, 30\}$, which are all multiples of $3$ that divide $30$. Observe that $\mathcal{N}_1(1)=\mathbf{N}_d^*$ but if $n\neq 1$ then $\mathcal{N}_1(n)$ is the set of all multiples and divisors of $n$, but not all $\mathbf{N}^*$. But, as $1$ is a divisor of all natural numbers, $\mathcal{N}_2(n)=\mathbf{N}_d^*$ for all $n\in \mathbf{N}_d^*$. The element $1$ is the minimum, but there is no maximal element.
\end{itemize} 
\end{example}

  We will also denote some commonly used ordered sets as follows:

\begin{itemize}
\item $\mathbf{n<}$: The poset with elements $\{0, 1, \ldots, n-1\}$ and order given by $0< 1< \ldots< n-1$.

\item $S$: the poset with elements in the set $S$ and order given by equality. This is the poset with elements on $S$ such that all elements are independent.\\
Example: $S=\{0, 1, 2\}$.
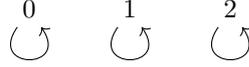
$$\begin{tikzcd}
0 \arrow[loop, distance=2em, in=305, out=235] & 1 \arrow[loop, distance=2em, in=305, out=235] & 2 \arrow[loop, distance=2em, in=305, out=235]
\end{tikzcd}$$

\item $\overline{S}$: the proset with elements in the set $S$ and order given by $i  \preceq j$ for all $i,j$. In other words, it is the proset where all elements are equivalent. In the case the set is $\{0, 1, 2, \ldots, n-1\}$ we will just denote it as $\textbf{n}$.\\
Example: $\textbf{3}$
$$
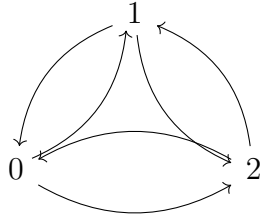
{\begin{tikzcd}
                                                             & 1 \arrow[rdd, bend right] \arrow[ldd, bend right] &                                                  \\
                                                             &                                                   &                                                  \\
0 \arrow[rr, bend right] \arrow[ruu, bend right, shift left] &                                                   & 2 \arrow[ll, bend right] \arrow[luu, bend right]
\end{tikzcd}}$$

\item $\mathbf{N}$: The natural numbers with the usual order, that is, $0< 1< 2< \ldots < n < \ldots$.

\item  $\mathbf{Z}$: The integers with the usual order, that is, $\ldots < -1< 0 < 1 < \ldots.$
\item $\mathbf{m}\leftarrow \mathbf{n}$: the set with elements $\{0', 1', \ldots, m-1', 0, 1, \ldots, n-1\}$ and order given by for every $i, j\in\{0' , 1', \ldots, m-1'\}$, for every $k, l\in\{0, 1, \ldots, n-1\}$ we have $i\sim j$, $i\precnapprox k$ and $k\sim l$.\\
Example: $\mathbf{2}\leftarrow \mathbf{2}$
$$% https://tikzcd.yichuanshen.de/#N4Igdg9gJgpgziAXAbVABwnAlgFyxMJZABgBpiBdUkANwEMAbAVxiRGIHIQBfU9TXPkIoAjOSq1GLNiK69+2PASIAmcdXrNWidjz4gMioUQDM6yVpk8JMKAHN4RUADMAThAC2SMiBwQkYiAARjBgUEgAtCbE8iBungHUft7UIWGR0bHxXohqvv6IIlnuOWb5SHlp4YhRMfrZFUkFZVUZMRTcQA
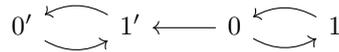
\begin{tikzcd}
0' \arrow[r, bend right] & 1' \arrow[l, bend right] & 0 \arrow[l] \arrow[r, bend right] & 1 \arrow[l, bend right]
\end{tikzcd}$$
Notice that the prosets of the form $\mathbf{n}$ can be described as $\mathbf{0}\leftarrow\mathbf{n}$, and the proset $\mathbf{2}<$ can be described as $\mathbf{1}\leftarrow\mathbf{1}$. The prosets of the form $\textbf{m}\leftarrow \textbf{n}$ are all the finite, irreducible prosets with at most two equivalence classes of elements.
\end{itemize}

\begin{definition}\label{defi:disjoint union prosets}
Let $(\Lambda,  \preceq)$ a proset. If there is a collection $\{\Lambda_i\}_{i\in I}$ of convex subsets of $\Lambda$ such that $\Lambda=\bigcup_{i\in I}\Lambda_i$, for every $i\neq j$ we have $\Lambda_i\bigcap \Lambda_j=\emptyset$ and for every $s_i\in \Lambda_i$, $s_j\in \Lambda_j$ we have that if $i\neq j$ then $s_i$ and $s_j$ are independent, we say $\Lambda$ is the \textbf{internal disjoint union} of $\{\Lambda_i\}_{i\in I}$. We denote it by $\Lambda=\bigsqcup_{i\in I} \Lambda_i$.\\
  If $\Lambda$ cannot be written as disjoint union of two non-empty subsets we say $\Lambda$ is \textbf{irreducible}.\index{disjoint union of POSETS|ndx} \index{irreducible POSET|ndx}
\   Let $\{(\Lambda_i,  \preceq_i)\}_{i\in I}$ a collection of prosets. We define the \textbf{external disjoint union} of this collection as the proset $(\bigsqcup_{i\in I} \Lambda_i,  \preceq_{\sqcup})$ with elements in $\bigsqcup_{i\in I}\Lambda_i$ and order $  \preceq_\sqcup$ given by $s_1  \preceq_\sqcup s_2\in \bigsqcup_{i\in I}\Lambda_i$ if, and only if, there is $i\in I$ such that $s_1  \preceq_i s_2\in \Lambda_i$.
\end{definition}

For the ones familiar with category theory, the disjoint union is the coproduct on the category of prosets. Given the posets $\mathbf{2}\! <=\{0, 1\}$ and $\mathbf{3}\!<=\{0', 1', 2'\}$ then $\mathbf{2\!<}\!\sqcup\! \mathbf{3\!<}$ is the  poset with elements $\{0, 1, 0', 1', 2'\}$ and Hasse diagram as follows:
$$
% https://tikzcd.yichuanshen.de/#N4Igdg9gJgpgziAXAbVABwnAlgFyxMJZABgBpiBdUkANwEMAbAVxiRACYByEAX1PUy58hFGQCMVWoxZsx3PgOx4CRMu0n1mrRCGLz+IDEuFEx5DdO0gxvA0aEqUZ9dU0ydxXpJhQA5vCJQADMAJwgAWyQAZmocCCQAFgUQUIikMhA4pDFk1MjEM0z4xHYeCh4gA
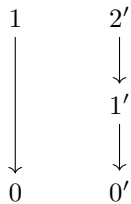
\begin{tikzcd}
1 \arrow[dd] & 2' \arrow[d]   \\
 & 1' \arrow[d]             \\
0           & 0'           
\end{tikzcd}
$$

  The main use of disjoint union is breaking down ordered sets into irreducible parts, as shown below:

\begin{proposition}\label{prop:making prosets into disjoint parts}
Let $\Lambda$ be a non-empty preordered set. Then there is a collection $\{\Lambda_i\}_{i\in I}$ of irreducible subsets of $\Lambda$ such that $\Lambda=\bigsqcup_{i\in I} \Lambda_i$.
\end{proposition}

\begin{proof}
Notice that for every $s\in \Lambda$, every element $t$ on the set $\mathcal{N}_{\omega}(s)$ is independent from every element $t'$ on the set $\Lambda\backslash \mathcal{N}_{\omega}(s)$. It is also the case that for every $s\in\Lambda$, the set $\mathcal{N}_{\omega}(s)$ is irreducible.\\
  Now let $I\subset \Lambda$ a maximal subset of pairwise independent elements of $\Lambda$ (such set exists by Zorn's Lemma). Define for each $i\in I$ the irreducible subset $\Lambda_i:=\mathcal{N}_{\omega}(i)$. As the set $I$ is maximal, the observation above implies 
$$\Lambda=\bigsqcup_{i\in I}\Lambda_i.$$
%First we will prove, using Zorn's Lemma, that there is a maximal collection of independent elements in $\Lambda$. With those we will build our irreducibles subsets and prove they form the entire set.\\
%  Consider $\mathcal{A}$ the collection of all independent subsets of $\Lambda$ with order given by $\subset$. As $\Lambda$ is non-empty then $\mathcal{A}$ contains any one element subset of $\Lambda$, hence $\mathcal{A}$ is a non-empty collection.\\
%  Given $\{A_i\}_{i\in I}$ an increasing collection of elements of $\mathcal{A}$. Define $A=\bigcup_{i\in I} A_i$. By definition $A_i\subset A$ for all $i\in I$, hence it is only necessary to prove $A\in\mathcal{A}$. Let $s, t\in A$. As the collection is increasing, there exists $i_0\in I$ such that $s, t\in A_{i_0}$. As all the elements in $A_{i_0}$ are independent it follow that $s$ and $t$ are independent. It follows then that $A\in \mathcal{A}$ and that every increasing sequence of elements in $\mathcal{A}$ has an upper bound. By Zorn's Lemma there exists a maximal element in $\mathcal{A}$.  \\
%  Let $\{s_i\}_{i\in I}$ a maximal collection of independent elements of $\Lambda$ and define $\Lambda_i=\mathcal{N}_\omega(s_i)$. By the definition of two elements being independent we have that $\Lambda_i\bigcap \Lambda_j=\emptyset$ for every $i\neq j$. To prove $\Lambda=\bigsqcup_{i\in I} \Lambda_i$ observe that if there was $q\in \Lambda\backslash \bigsqcup_{i\in I} \Lambda_i$ we would get a contradiction on the maximality of $\{s_i\}_{i\in I}$.
\end{proof}

We call the irreducible subsets $\Lambda_i=\mathcal{N}_{\omega}(i)$ the \textbf{components} of $\Lambda$.% The irreducible sets also give us a finite convex hull for finite subsets of prosets as seen below.

\begin{proposition}\label{prop:finite subsets into finite convex}
Let $\Lambda$ an irreducible locally finite proset and $S\subset \Lambda$ a finite subset. There is a finite convex subset $\alpha\subset \Lambda$ such that $S\subset \alpha$.
\end{proposition}

\begin{proof}
The result is trivial for $S$ the empty set. If $|S|>0$, let $s\in S$. For each $s'\in S$ let $t_1^{s'}, t_2^{s'}, \ldots, t_n^{s'}$ be elements such that $t_1^{s'}\in \mathcal{N}_1(s)$, $t^{s'}_{i+1}\in \mathcal{N}_1(t^{s'}_i)$, $s'\in \mathcal{N}_1(t^{s'}_n)$. As $\Lambda$ is irreducible such elements exist. Let $S'=S\bigcup\left(\cup_{s'\in S\backslash\{s\}}\{t_1^{s'}, t_2^{s'}, \ldots, t_n^{s'}\}\right)$. This set is a finite union of finite sets, hence is finite. Let $\alpha=\bigcup_{s_1\leqslant s_2, s_1, s_2\in S'}[s_1, s_2]$. This set is a finite union of finite intervals hence it is also finite. \\
  To see $\alpha$ is convex, let $t_1, t_2\in \alpha$ be such that $t_1  \preceq t_2$ in $\Lambda$. By the definition of $\alpha$ there are $s_1, s_2, s_1', s_2'\in S'$ such that $t_1\in [s_1, s_2]$ and $t_2\in[s_1', s_2']$. Notice that $s_1  \preceq t_1  \preceq t_2  \preceq s_2'$, hence $[t_1, t_2]\subset [s_1, s_2']$. That is, $\alpha$ is closed under intervals. By definition it is also convex. 
\end{proof}

By looking at the components that intersect a finite subset $S\subset \Lambda$ individually, we obtain the following corollary:

\begin{corollary}\label{coro:finite sets into finite locally convex}
Given $\Lambda$ a locally finite proset and $S\subset \Lambda$ a finite subset, there is a locally convex collection $\{\Lambda_i\}_{i\in I}$ such that $S\subset \bigcup_{i\in I}\Lambda_i$ and $\bigcup_{i\in I}\Lambda_i$ is finite.
\end{corollary}

\begin{definition}
Let $f:\Lambda_1\rightarrow \Lambda_2$ a function between preordered sets. We say $f$ is an \textbf{bijective order preserving map} if it is a set bijection, and for every $s_1, s_2\in \Lambda_1$ we have that $s_1\preceq s_2$ if, and only if, $f(s_1)\preceq f(s_2)$.
\end{definition}

\subsection{The units group of a 
matrix ring}

Given $P$ a ring, we denote the \textbf{group of units of $P$} as $P^*$. We will denote the ring of $n\times n$ matrices over $P$ by $\textup{M}_n(P)$, the group of $n\times n$ invertible matrices by $\textup{GL}_n(P)$, and the group of $n\times n$ invertible matrices with determinant $1$ as $\textup{SL}_n(P)$.

\begin{theorem}
The functor that maps a ring $P$ to its group of units is right adjoint. 
\end{theorem}

\begin{definition}\label{thrm:giving topology to unit group}
An \textbf{absolute topological ring} is a topological ring such that its group of units with the subspace topology is a topological groups.
\end{definition}

As topological fields are defined so the inverse map is continuous, every topological field is an absolute topological ring.

%\begin{lemma}[§§103–105 \cite{dickson1901linear}]\label{Lemma:how to make the simple matrices}
%Let $F$ be a field and let $n\in \mathbb{N}$ such that $n\geqslant 2$ ; in the case $n = 2$, assume $|F| > 3$. Then every proper normal subgroup $\textup{SL}_n(F)$ is central, and every noncentral normal subgroup of $\textup{GL}_n(F)$ contains the groups $\textup{SL}_n(F)$.
%\end{lemma}

\begin{theorem}\label{thrm:conditions invertible}
Let $P$ be a commutative ring. Then for every $A\in \textup{M}_n(P)$ we have
$$A \textup{adj}(A)=\textup{adj}(A)A=\det(A)I_n.$$
Hence, $A\in \textup{M}_n(P)$ is invertible if, and only if, $\det(A)$ is invertible in $P$.
\end{theorem}

\begin{theorem}\label{thrm:invertible as product}
Let $P$ be a commutative ring. If $A\in \textup{GL}_n(R)$ then
$$A^{-1}=\frac{1}{\det(A)}\textup{adj}(A).$$
\end{theorem}

\begin{theorem}[Cayley-Hamilton Theorem]\label{thrm:cayley-hamilton}
Let $P$ be a commutative ring. If $A\in \textup{GL}_n(R)$ then
$$\textup{adj}(A)=-\sum_{i=1}^{n}c_iA^{i-1}$$
where $c_i$ are the coefficients of the characteristic polynomial of $A$.
\end{theorem}

\section{Topological incidence rings}\label{Topological incidence rings}

\subsection{First definitions}

\begin{definition}[Definition 1.1 \cite{Matrices1973}]\label{defi:proset matrices}\index{$\textup{M}_{\Lambda}(P)$|ndx}\index{ring of matrices indexed by a proset|ndx}
Let $P$ be a ring and $(\Lambda,  \preceq)$ a locally finite proset. We define the \textbf{incidence ring of $\Lambda$ over $P$} as
$$\textup{M}_\Lambda(P):=\left\{ (a_{s_1,s_2})\in P^{\Lambda\times \Lambda}
: \textit{ if }a_{s_1,s_2}\neq 0\textit{ then } s_1  \preceq s_2 \right\}$$
with  coordinatewise sum and multiplication defined as follows: given two element $(a_{s_1,s_2}),(b_{s_1,s_2})\in \textup{M}_\Lambda(P)$ then $(a_{s_1,s_2}).(b_{s_1,s_2})=(c_{s_1,s_2})$, where 
$$c_{s_1,s_2}=\sum_{t\in[s_1,s_2]}a_{s_1,t}b_{t,s_2}.$$ 
\end{definition}

Note that this multiplication is defined in such a way that for any finite $n\times n$ block of the matrix centered on the main diagonal, the multiplication follows the same rule as $\textup{M}_\mathbf{n}(P)$. We will denote the elements of $\textup{M}_\Lambda(P)$ as $A,B,C$. We denote $A_{s_1,s_2}$ as the coordinate $(s_1, s_2)$ in the matrix $A$. When talking about the coordinate of the product $A_1A_2\ldots A_n$ we denote it as $(A_1A_2\ldots A_n)_{s_1,s_2}$. Following are some examples of incidence rings.

\begin{example}
\begin{enumerate}
    \item Given $n\in\mathbb{N}$, the ring $\textup{M}_{\mathbf{n}}(P)$ is the ring of $n\times n$ matrices over the ring $P$ with the usual multiplication.
    \item Given $n\in\mathbb{N}$, the ring $\textup{M}_{\mathbf{n<}}(P)$ is the ring of $n\times n$ upper triangular matrices over $P$ with the usual multiplication.
    \item Given $S$ a set, the ring $\textup{M}_S(P)$ is isomorphic to $\prod_{s\in S} P$ with the coordinate-wise multiplication. 
    \item The ring $\textup{M}_{\mathbf{N}^*_d}(P)$ has elements of the form:

$$
A:=\left(\begin{array}{cccccc}
A_{1,1} & A_{1,2} & A_{1,3} & A_{1,4} & A_{1,5}  & \cdots \\
0 & A_{2,2} & 0 & A_{2,4} & 0 & \cdots\\
0 & 0 & A_{3,3} & 0 & 0  & \cdots \\
0 & 0 & 0 & A_{4,4} & 0 &  \cdots \\
0 & 0 & 0 & 0 & A_{5,5} &  \cdots  \\
0 & 0 & 0 & 0 & 0 &  \cdots \\
\vdots & \vdots  & \vdots & \vdots & \vdots & \ddots
\end{array}\right)
$$

\noindent and the multiplication of two matrices $A,B\in \textup{M}_{\mathbf{N}^*_d}(P)$ is:

$$
A
B:=\left(\begin{array}{ccccccc}
A_{1,1}B_{1,1} & \sum_{d|2}A_{1,d}B_{d,2} & \sum_{d|3}A_{1,d}B_{d,3} & \sum_{d|4}A_{1,d}B_{d,4} & \sum_{d|5}A_{1,d}B_{d,5}  & \cdots \\
0 & A_{2,2}B_{2,2} & 0 & \sum_{d|4}A_{2,d}B_{d,4} & 0 & \cdots\\
0 & 0 & A_{3,3}B_{3,3} & 0 & 0 &  \cdots \\
0 & 0 & 0 & A_{4,4}B_{4,4} & 0 &  \cdots \\
0 & 0 & 0 & 0 & A_{5,5}B_{5,5} &  \cdots  \\
0 & 0 & 0 & 0 & 0 &  \cdots \\
\vdots & \vdots  & \vdots & \vdots & \vdots & \ddots
\end{array}\right)
$$
\end{enumerate}
\end{example}
Given $S\subset \textup{M}_{\Lambda}(P)$ a subset, we define $S_{s_1,s_2}=\{A_{s_1,s_2}: A\in S\}\subset P$. Note though that $R\subset \textup{M}_\Lambda(P)$ a subset is a subring if $0,1\in R$ and for every $s_1  \preceq s_2$ in $\Lambda$ we have that
$$R_{(s_1,s_2)}+R_{(s_1,s_2)}\subset R_{(s_1,s_2)},$$
$$\sum_{s\in[s_1,s_2]}R_{s_1,s}R_{s,s_2}\subset R_{s_1,s_2}.$$
And $I\subset \textup{M}_{\Lambda}(P)$ a subset is a two-sided ideal if $0\in I$ and for every $s_1  \preceq s_2$ in $\Lambda$ we have that
$$R_{(s_1,s_2)}+R_{(s_1,s_2)}\subset R_{(s_1,s_2)},$$
$$\sum_{s\in[s_1,s_2]}I_{s_1,s}R_{s,s_2}\subset I_{s_1,s_2},$$
$$\sum_{s\in[s_1,s_2]}R_{s_1,s}I_{s,s_2}\subset I_{s_1,s_2}.$$
This allows us to easily describe some ideals and subrings of $\textup{M}_{\Lambda}(P)$ central to our work. These also allow us to prove the existence of a topology on $\textup{M}_\Lambda(P)$ in a natural way, preserving properties of the finite blocks of matrices.\\
\textbf{Notation}:[Ideals of $\textup{M}_\Lambda(P)$]\label{defi:ideals}
Let $\Lambda$ a locally finite proset and $P$ a ring. We denote the following two-sided ideals:
\begin{itemize}
    \item Given $s_1  \preceq s_2$ in $\Lambda$ we define 
    $$I^{\Lambda}_{[s_1, s_2]}(P)=\{A\in \textup{M}_{\Lambda}(P): \text{ if } t_1, t_2\in[s_1, s_2] \text{ then } A_{t_1,t_2}=0 \}.$$ 
    This ideal has quotient 
    $$\textup{M}_{\Lambda}(P)/I^{\Lambda}_{[s_1,s_2]}(P)\cong \textup{M}_{[s_1,s_2]}(P).$$
    \item For a convex set $\Lambda'\subset \Lambda$, the ideal
    $$I^{\Lambda}_{\Lambda'}=\bigcap_{s_1  \preceq s_2\in \Lambda'}I^{\Lambda}_{[s_1,s_2]}.$$
    This ideal has quotient $$\textup{M}_{\Lambda}(P)/I^{(\Lambda,  \preceq)}_{\Lambda'}\cong \textup{M}_{\Lambda'}(P).$$
    \item For $\{\Lambda_i\}_{i\in I}$ a locally convex set of $\Lambda$ we define the ideal 
    $$I^{\Lambda}_{\{\Lambda_i\}_{i\in I}}=\bigcap_{i\in I}I^{\Lambda}_{\Lambda_i}.$$
    This ideal has quotient 
    $$\textup{M}_{\Lambda}(P)/I^{\Lambda}_{\{\Lambda_i\}_{i\in I}}\cong \prod_{i\in I} \textup{M}_{\Lambda_i}(P).$$
    \item Given $J\leqslant P$ a two-sided ideal we define
    $$\textup{M}_{\Lambda}(J):=\{A\in \textup{M}_\Lambda(P): \text{ for every } t_1,t_2\in \Lambda \text{ then } A_{t_1,t_2}\in J \}\leqslant \textup{M}_\Lambda(P).$$
    This ideal has quotient 
    $$\textup{M}_\Lambda(P)/\textup{M}_\Lambda(J)\cong \textup{M}_\Lambda(P/J).$$
    \item  Given $\{\Lambda_i\}_{i\in I}$ a locally convex set of $\Lambda$ and $J\leqslant P$ a two-sided ideal of $P$, then the ideal $$I_{\{\Lambda_i\}_{i\in I}}+\textup{M}_\Lambda(J)$$
    has quotient 
    $$\textup{M}_\Lambda(P)/(I_{\{\Lambda_i\}_{i\in I}}+\textup{M}_\Lambda(J))\cong \prod_{i\in I} \textup{M}_{\Lambda_i}(P/J).$$
\end{itemize} 
Some of the ideals above were initially defined at \cite{Matrices1973}.

\begin{example}
In the case of the matrix ring $\textup{M}_\mathbf{N}(\mathbb{Z}_p)$ and interval be $[0, 2]=\{0, 1, 2\}$, we have that $I_{[0,2]}$ has elements of the form
$$
\left(\begin{array}{cccccc}
0 & 0 & 0 & A_{0,3} & A_{0,4}  & \cdots \\
0 & 0 & 0 & A_{1,3} & A_{1,4} & \cdots\\
0 & 0 & 0 & A_{2,3} & A_{2,4}  & \cdots \\
0 & 0 & 0 & A_{3,3} & A_{3,4} &  \cdots \\
0 & 0 & 0 & 0 & A_{4,4} &  \cdots  \\
\vdots & \vdots  & \vdots & \vdots & \vdots & \ddots
\end{array}\right)
$$
and the quotient $\textup{M}_\mathbf{N}(\mathbb{Z}_p)/I_{[0,2]}$ is isomorphic to $\textup{M}_{\mathbf{3<}}(\mathbb{Z}_p)$, the ring of upper $3\times 3$ triangular matrices over $\mathbb{Z}_p$. Given the ideal $p\mathbb{Z}_p$ of $\mathbb{Z}_p$, the quotient $\textup{M}_\mathbf{N}(\mathbb{Z}_p)/\textup{M}_\mathbf{N}(p\mathbb{Z}_p)$ is isomorphic to $\textup{M}_\mathbf{N}(\mathbb{Z}/p\mathbb{Z})$, and  $\textup{M}_\mathbf{N}(\mathbb{Z}_p)/(I_{[0,2]}+\textup{M}_\mathbf{N}(p\mathbb{Z}_p))$ is isomorphic to $\textup{M}_{\mathbf{3<}}(\mathbb{Z}/p\mathbb{Z}).$
\end{example}

\subsection{Infinite matrices as limits of finite matrices}

To prove the results in this section we will look at $\textup{M}_\Lambda(P)$ as objects in the category of associative topological rings with identity and the morphisms are the continuous homomorphism of rings. This category will be denoted as $\textbf{TopRings}$. For the next result, notice that if $P$ is a topological ring and $\Lambda$ is a finite proset, then the ring $\textup{M}_\Lambda(P)$ is a topological ring with subspace topology in relation to the product space $P^{\Lambda\times\Lambda}$. In a similar way, for $\Lambda$ an infinite, locally finite proset we define the topology on $\textup{M}_{\Lambda}(P)$ as the subspace topology on $P^{\Lambda\times \Lambda}$. Our main goal in this subsection is proving that with this topology, $\textup{M}_{\Lambda}(P)$ is a topological ring.\\
  Notice that if $\Lambda'\subset \Lambda$ is a finite convex subset, say $\Lambda'=\{s_1, s_2, \ldots, s_n\}$, then $\textup{M}_{\Lambda'}(P)$ can be seen as a subring of $\textup{M}_{\mathbf{n}}(P)$ with embedding determined by the enumeration of $\Lambda'$. The ideals of the form $I^\Lambda_{\Lambda'}$, for $\Lambda'$ a finite convex subset, will then be essential to build our inverse limit and the topology, as the quotient under such ideals are topological incidence rings.\\
  Given $\Lambda$ a proset, define $\Gamma(\Lambda)=\{\alpha\subset \Lambda:  \alpha\text{ is finite and convex}\}$ \textbf{the set of all finite convex subsets of $\Lambda$}. The next result follows directly from Proposition \ref{prop:finite subsets into finite convex}.\index{$\Gamma(\Lambda)$|ndx}\index{set of finite, convex subsets of a proset|ndx}

\begin{corollary}\label{coro:directed set}
If $\Lambda$ is an irreducible poset then for convex subsets $\alpha, \beta\in \Gamma(\Lambda)$ there is $\gamma\in \Gamma(\Lambda)$ such that $\alpha\subset \gamma$ and $\beta\subset \gamma$, that is, under the $\subset$ order the set $\Gamma(\Lambda)$ is a directed set.
\end{corollary}
%\begin{proof}
%By Proposition \ref{prop:finite subsets into finite convex} it follows that, because $\alpha, \beta$ are finite there exists $\gamma$ such that $\gamma$ is convex and $\alpha,\beta\subset \gamma$.
%\end{proof}

%Notice that this definition implies that the quotient of such matrix rings under $I_{\Lambda'}$ (Notation \ref{defi:dense subset}) is continuous. Some examples of rings satisfying these properties are any discrete rings, locally compact topological fields and prodiscrete rings.

We can then prove the following Theorem:
\begin{theorem}\label{Prop: prodiscrete ring}
Let $P$ a ring and $(\Lambda,  \preceq)$ an infinite irreducible locally proset. Then $\textup{M}_\Lambda(P)\cong\varprojlim_{\alpha\in\Gamma(\Lambda)} \textup{M}_\alpha(P)$. If $P$ is a topological ring then $\textup{M}_{\Lambda}(P)$ is a topological ring with topology given by the inverse limit.
\end{theorem}

\begin{proof}

By Corollary \ref{coro:directed set} we have that $\Gamma(\Lambda)$ is a directed set. Given $\alpha\subset \beta\in \Gamma(\Lambda)$ define $\pi_{\beta\rightarrow\alpha}:\textup{M}_\beta(P)\longrightarrow \textup{M}_\alpha(P)$ as the natural projection of $\textup{M}_\beta(P)$ with kernel $I^{\beta}_{\alpha}$. It is easy to see that given $\alpha\subset\beta\subset\gamma$ then $\pi_{\gamma\rightarrow\alpha}=\pi_{\beta\rightarrow\alpha}\circ \pi_{\gamma\rightarrow\beta}$, that is, the following diagram commutes.\\
\begin{center}
\begin{tikzcd}
% https://tikzcd.yichuanshen.de/#N4Igdg9gJgpgziAXAbVABwnAlgFyxMJZARgBoAGAXVJADcBDAGwFcYkQAVAJywAoAdfgHN6AW1H0AlCAC+pdJlz5CKcqWLU6TVu259BAIxg4ps+SAzY8BIgCZ1mhizaJOPAfyZoAFqZmaYKCF4IlAAMy4IUSQ1EBwIJHstZ3ZBNCwAfWBBEXF6QR4hbxMuSIB3QS9fGRAaRnojRgAFRWsVECwwbFgzcMjoxFj4pDJknVc0zOzhMQkCrCKS8sNjehq6hphm1uV2Tu62OT6okZphxCSncZBJrJWTecX6UogKz0YfNdqQesaWq12rn2WB6-hkQA
                                                                  & \textup{M}_\gamma(P) \arrow[rd, "\pi_{\gamma\rightarrow\alpha}" description] \arrow[ld, "\pi_{\gamma\rightarrow\beta}" description] &             \\
\textup{M}_\beta(P) \arrow[rr, "\pi_{\beta\rightarrow\alpha}" description] &                                                                                                                            & \textup{M}_\alpha(P)
\end{tikzcd}
\end{center}

The collection $((\textup{M}_{\alpha})_{\alpha\in \Gamma(\Lambda)},(\pi_{\beta\rightarrow\alpha})_{\alpha\subset\beta})$ is then an inverse system. As the category of associative rings with identity is closed under inverse limits, the inverse limit of this system exists and is unique. Denote $\varprojlim_{\alpha\in\Gamma(\Lambda)} \textup{M}_\alpha(P)$ as the inverse limit of this system. Our objective now is proving $\textup{M}_\Lambda\cong \varprojlim_{\alpha\in\Gamma(\Lambda)} \textup{M}_\alpha(P)$.\\
  Observe that for every $\alpha\in\Gamma(\Lambda)$ there is $\psi_{\alpha}:\textup{M}_\Lambda(P)\longrightarrow \textup{M}_\alpha(P)$ given by the natural projection of $\textup{M}_\Lambda$ under $I^\Lambda_{\alpha}$, and these projections are such that the outer triangle of the following diagram commutes.
\\

\begin{center}
% https://tikzcd.yichuanshen.de/#N4Igdg9gJgpgziAXAbVABwnAlgFyxMJZAJgBoAGAXVJADcBDAGwFcYkQAVAJywAoAdfgBl6AWwBGUegEoQAX1LpMufIRRkAjNTpNW7APTzFIDNjwEiZYtoYs2iEIIZc0XCACtGWUQH1ggpjQAC3pBLDBBAHExUXoBYTFJGTkAAm4+AMZgmSMlM1UiclIAZhtde04eeMCQn2JZBTyVCxQAFhKyu3Z06qzajQbtGCgAc3giUAAzN1EkIpAcCCQyHS6HSZAaRnpxGEYABWVzNRBw7FhNkC8wCqk4IOHckGmIWcR5xaR21b0HQTRsH5MtkfBo5JdtrsDkcCg4zlgLo1njM5jRPohijRbL9HPwAVggfwavQ6uCtjs9od8i1TmBzmwkS83it0d9sRV-gT-ES+iSwRCKdDqSd4YjjEzlmilhiseV2JzCcTSQKoVTmiK6QiGeKUTKFtK2XK-niucDasRBDwRkEcPQuG4AO5mvlkq6CtXHdiihmUORAA
\begin{tikzcd}
                                                                            &  & \textup{M}_\Lambda(P) \arrow[dd, "f" description, dashed] \arrow[rrddd, "\psi_{\alpha_1}" description] \arrow[llddd, "\psi_{\alpha_2}" description] &  &               \\
                                                                            &  & /                                                                                                                                          &  &               \\
                                                                            &  & \varprojlim_{\alpha\in\Gamma(\Lambda)} \textup{M}_\alpha(P) \arrow[rrd, "\pi_{\alpha_1}" description] \arrow[lld, "\pi_{\alpha_2}" description]     &  &               \\
\textup{M}_{\alpha_2}(P) \arrow[rrrr, "\pi_{\alpha_2\rightarrow\alpha_1}" description] &  &                                                                                                                                            &  & \textup{M}_{\alpha_1}(P)
\end{tikzcd}
\end{center}

By the universal property of inverse limits, there exists an unique ring homomorphism $f:\textup{M}_\Lambda(P)\longrightarrow \varprojlim_{\alpha\in\Gamma(\Lambda)} \textup{M}_\alpha(P)$ such that the whole diagram commutes. It remains to show the function $f$ is a bijection.\\
  Let $A\in \textup{M}_\Lambda(P)$ be such that $f(A)=0$. Because the diagram commutes we have that $\psi_{\alpha}(A)=0$ for all $\alpha\in\Gamma(\Lambda)$. Because these projections are such that $\psi_{\alpha}(A)=0$ if, and only if, $A_{t_1,t_2}=0$ for all $t_1, t_2\in\alpha$ we have that $A=0$. Hence $f$ is injective.\\
  It remains to prove that $f$ is surjective. For that, given $\overline{A}\in \varprojlim_{\alpha\in\Gamma(\Lambda)} \textup{M}_\alpha(P)$ we need to find $A\in \textup{M}_\Lambda(P)$ such that $f(A)=\overline{A}$. We will describe such $A$ coordinate-wise. Notice that as $f$ is a ring homomorphism, it will an isomorphism if, and only if, $f$ is an isomorphism of abelian groups (when we forget the multiplication structure on our ring). Hence we can treat these as abelian groups to simplify our proof.\\
  Observe that by looking at these as abelian groups, for every proset $\Lambda$ we have that $\textup{M}_\Lambda(P)$ is a subgroup of $P^{\Lambda\times \Lambda}$. Hence for every $s_0, t_0\in \Lambda$ there is a projection $$\begin{array}{cccc}
\pi^\Lambda_{s_0,t_0}: & \textup{M}_\Lambda(P) & \rightarrow & P\\
 & A & \mapsto & A_{s_0,t_0}.
\end{array}$$ 
Similarly, for $\alpha_1,\alpha_2\in\Gamma( \Lambda)$, with $\alpha_1\subset\alpha_2$, and $s_0, t_0\in\alpha_1$ there are projection $\pi^{\alpha_1}_{s_0,t_0}:\textup{M}_{\alpha_1}(P)\rightarrow P$ and $\pi^{\alpha_2}_{s_0,t_0}:\textup{M}_{\alpha_2}(P)\rightarrow P$ such that $\pi_{s_0,t_0}^{\alpha_1}\circ \psi_{\alpha_1}=\pi_{s_0,t_0}^{\alpha_2}\circ \psi_{\alpha_2}$, and the  following diagram commutes:
\begin{center}
% https://tikzcd.yichuanshen.de/#N4Igdg9gJgpgziAXAbVABwnAlgFyxMJZAJgBoAGAXVJADcBDAGwFcYkQAVAJywAoAdfgBl6AWwBGUegEoQAX1LpMufIRRkAjNTpNW7APTzFIDNjwEiZYtoYs2iEIIZc0XCACtGWUQH1ggpjQAC3pBLDBBAHExUXoBYTFJGTkAAm4+AMZgmSMlM1UiclIAZhtde04eeMCQn2JZBTyVCxQAFhKyu3Z06qzajQbjU2a1ElJWzr0HAAV5bRgoAHN4IlAAMzdRJCKQHAgkMh0uhzWQGkZ6cRhGaeVzNRBw7FgzkC8wCqk4IIXckA2IFtEDs9kh2kcpo5+GhsH5MtkfBo5K8LlcbncCg4nlgXo1-pttjRQYhijRbJDBDCsHD+DV6HVkedLtdbvkWo8wM82HiAUDDsTweSKpTqf5aX16UiUcz0WyHtjccZeQcifsSWTyuwRTS6QzpWjWSN2AruUqCerdmrBZqHNqxbriIIeIsgjh6Fw3AB3eH9RlvGWG+7Gzk403rc2ky1IACsTINGPZJteQq10KwAD17RKGX44D4ijh88ieebwcTY-743Lg1zkzaoWgM1mEUjc-nSIXyMXKHIgA
\begin{tikzcd}
                                                                                                                                  &  & \textup{M}_\Lambda(P) \arrow[dd, "f" description, dashed] \arrow[rrddd, "\psi_{\alpha_1}" description] \arrow[llddd, "\psi_{\alpha_2}" description] &  &                                                                    \\
                                                                                                                                  &  & /                                                                                                                                          &  &                                                                     \\
                                                                                                                                  &  & \varprojlim_{\alpha\in\Gamma(\Lambda)} \textup{M}_\alpha(P) \arrow[rrd, "\pi_{\alpha_1}" description] \arrow[lld, "\pi_{\alpha_2}" description]     &  &                                                                     \\
\textup{M}_{\alpha_2}(P) \arrow[rrrr, "\pi_{\alpha_2\rightarrow\alpha_1}" description] \arrow[rrd, "{\pi^{\alpha_2}_{s_0,t_0}}" description] &  &                                                                                                                                            &  & \textup{M}_{\alpha_1}(P) \arrow[lld, "{\pi^{\alpha_1}_{s_0,t_0}}" description] \\
                                                                                                                                  &  & P                                                                                                                                          &  &                                                                    
\end{tikzcd}
 \end{center}

\noindent That is, the coordinate $A_{s_0,t_0}$ is the same for all $\alpha\in\Gamma(\Lambda)$ such that $s_0, t_0\in\alpha$.\\
  Notice that $\pi_{s_0,t_0}^{\beta_1}\circ\pi_\alpha(\overline{A})=\pi_{s_0,t_0}^{\beta_2}\circ\pi_\alpha(\overline{A})$ for every $\beta_1, \beta_2\in \Gamma(\Lambda)$ such that $[s_0, t_0]\subset \beta_i$, $i=1, 2$. We define $A$ coordinatewise as $A_{s_0,t_0}=\pi^{\alpha}_{s_0,t_0}\circ\pi_{\alpha}(\overline{A})$, for $s_0, t_0\in \alpha$ and $\alpha\in\Gamma(\Lambda)$. As the diagram commutes, $\pi^\alpha_{s_0,t_0} \circ \pi_{\alpha} \circ f=\pi^\alpha_{s_0,t_0}\circ\psi_{\alpha}$ for all $s_0  \preceq t_0$ such that $s_0, t_0\in \alpha$, and $\alpha\in\Gamma(\Lambda)$. Hence $A$ is so that $f(A)=\overline{A}$. That is, $f$ is an isomorphism of rings.\\
  By the inverse limit property, $\textup{M}_\Lambda(P)$ is a closed subset of $\prod_{\alpha\in \Gamma(\Lambda)}\textup{M}_{\alpha}(P)$ under the product topology, hence it can be seen as a topological ring.\\
  Notice that, as observed earlier for finite prosets, it then follows that $\textup{M}_\Lambda(P)$ is a closed subset of $\prod_{\alpha\in\Gamma(\Lambda)}P^{\alpha\times \alpha}$. As $\Lambda$ is infinite, we have that $|\Gamma(\Lambda)|=|\Lambda|$. Hence there exists a homeomorphism between $\prod_{\alpha\in\Gamma(\Lambda)}P^{\alpha\times \alpha}$ and $P^{\Lambda\times \Lambda}$ given by a bijection from $\bigsqcup_{\alpha\in \Gamma{(\Lambda)}}\alpha\!\times\! \alpha$ to $\Lambda\!\times\!\Lambda$. That is, the topological ring structure on $\textup{M}_\Lambda(P)$ can be seen as the subspace topology in relation to $P^{\Lambda\times \Lambda}$, as in the finite proset case.
\end{proof}

\begin{proposition}[Lemma 2.2 \cite{Matrices1985}]\label{prop:disjoint union and product rings}
Let $P$ be a ring and $\Lambda=\sqcup_{i\in I}\Lambda_i$ a preordered set. Then $\textup{M}_{\Lambda}(P)\cong \prod_{i\in I}\textup{M}_{\Lambda_i}(P)$. If $P$ is a topological ring then $\textup{M}_{\Lambda}(P)$ is a topological ring with topology given by the product topology on the $\textup{M}_{\Lambda_i}(P)$.
\end{proposition}

%\begin{proof}
%Let $A\in \textup{M}_{\Lambda}(P)$. Define, for each $i\in I$ the surjective homomorphisms $\pi_i:\textup{M}_{\Lambda}(P)\rightarrow \textup{M}_{\Lambda_i}(P)$ coordinatewise as $\pi_i(A)_{s_1,s_2}:=  A_{s_1,s_2}$ if $s_1, s_2\in \Lambda_i$. Given $R$ a ring and, for every $i\in I$, a map $f_i:R\rightarrow \textup{M}_{\Lambda_i}(P)$ one can define $f:R\rightarrow \textup{M}_\Lambda(P)$ coordinatewise as:
%$$f(r)_{s_1,s_2}:=\left\{\begin{array}{ll}
%   f_i(r)_{s_1,s_2}  & \mbox{if }s_1, s_2\in \Lambda_i  \\
 %   0 & \mbox{otherwise.}
%\end{array}\right.$$
%\noindent This is clearly a ring homomorphism and satisfy $f_i=\pi_i\circ f$ for every $i\in I$. Hence the result follows.
%\end{proof}
Notice that Theorem \ref{Prop: prodiscrete ring} and Proposition \ref{prop:disjoint union and product rings} implies that all matrix rings indexed by prosets are limits of finite matrix rings. Both results then imply that the topology given to $\textup{M}_{\Lambda}(P)$ in relation to $\textup{P}^{\Lambda\times \Lambda}$ makes $\textup{M}_{\Lambda}(P)$ into a topological ring. More than that, this is the topology given by the inverse limit.\\
  For an example, let $\mathbf{2}\!<$ be our poset and $P$ be any ring. We can define an isomorphism of rings from $f:\textup{M}_{\mathbf{2}<\sqcup \mathbf{2}<}(P)\rightarrow \textup{M}_{\mathbf{2}<}(P)\times \textup{M}_{\mathbf{2}<}(P)$ as follows:
$$\left(\begin{array}{cccc}
        A_{0,0} & A_{0,1} & 0 & 0 \\
        0 & A_{1,1} & 0 & 0 \\
    0 & 0 & A_{3,3} & A_{3,4} \\
        0 & 0 & 0 & A_{4,4}
\end{array}\right)\mapsto \left(\begin{array}{cc}
        A_{0,0} & A_{0,1} \\
        0 & A_{1,1} 
\end{array}\right)\times \left(\begin{array}{cc}
    A_{0',0'} & A_{0',1'} \\
        0 & A_{1',1'}
\end{array}\right).$$

%\begin{corollary}\label{Prop: prodiscrete ring actual result}
%Let $P$ a profinite/prodiscrete ring and $\Lambda$ an irreducible proset with . Then $\textup{M}_\Lambda(P)$ is also profinite/prodiscrete.
%\end{corollary}

%\begin{proof}
%Follows from the fact that products and inverse limits of profinite/prodiscrete rings are also profinite/prodiscrete.
%\end{proof}

\subsection{Topological structure}
We now use Theorem \ref{Prop: prodiscrete ring} and Proposition \ref{prop:disjoint union and product rings} to prove properties of the rings as topological spaces. Given $P$ a topological ring, $\Lambda$ a proset and $S\subset P$ a non-empty subset such that $0\in S$, we define $$\textup{M}_\Lambda(S):=\left\{ A\in \textup{M}_\Lambda(P): A_{s_1,s_2}\in S \text{ for all }s_1, s_2\in \Lambda \right\}\subset \textup{M}_\Lambda(P).$$
\begin{remark}\label{remark:nets and open}
    Let $\{X_i\}_{i\in I}$ be a collection of topological spaces. Define $X=\prod_{i\in I}X_i$ a topological space given by the product topology on the sets $X_i$. For each $i\in I$, let $\pi_i:X\rightarrow X_i$ be the projection on the $i$-th coordinate. Then:
    \begin{itemize}
        \item The subsets $O\subset X$ such that for all $i\in I$ the set $\pi_i(O)$ is an open subset of $X_i$, and there is a finite subset $J\subset I$ such that if $i\in I\backslash J$ then $\pi_i(O)=X_i$ form a basis of open subsets of $\prod_{i\in I}X_i$.
        %\item $C\subset X$, $C\neq \emptyset$ is closed if, and only if, $C=\prod_{i\in I}C_i$, and for all $i\in I$ $C_i\neq \emptyset$ is closed in $X_i$.
        \item Given $\{x_a\}_{a\in A}$ a net in $X$ converges, it converges to $x$ if, and only if, for every $i\in I$ the net $\{\pi_i(x_a)\}_{a\in A}$ converges to $\pi_i(x)$ in $X_i$. 
    \end{itemize}
\end{remark}
The following results follow directly from Remark \ref{remark:nets and open} and the topology given on $\textup{M}_\Lambda(P)$ as a topological subspace of $P^{\Lambda\times \Lambda}$.

\begin{corollary}\label{coro:classifying basis}
Let $P$ be a topological ring, $\{S_\sigma\}_{\sigma\in\Sigma}$ a basis of the neighbourhood of $0$ and $\Lambda$ a locally proset. Then $\{I^\Lambda_{\alpha}(P)+\textup{M}_\Lambda(S_\sigma)\}_{\alpha\in \Gamma(\Lambda), \sigma\in\Sigma}$ is a basis of the neighbourhood of $0$ in $\textup{M}_\Lambda(P)$.
\end{corollary}

\begin{corollary}\label{coro:open iff matrices}
Let $P$ be a topological ring and $\Lambda$ a locally finite proset. Given $J$ a two-sided open ideal of $P$ and $\{\Lambda_i\}_{i\in I}\subset \Lambda$ locally convex, then the following are true:
\begin{enumerate}
    \item The ideals $I^{\Lambda}_{\{\Lambda_i\}_{i\in I}}$ and $\textup{M}_\Lambda(J)$ are closed in $\textup{M}_\Lambda(P)$.
    \item The ideal $I^\Lambda_{\{\Lambda_i\}_{i\in I}}$ is open if, and only if, $\bigsqcup_{i\in I}\Lambda_i$ is finite and $P$ has the discrete topology.
    \item The ideal $\textup{M}_\Lambda(J)$ is open if, and only if, $\Lambda$ is finite.
\end{enumerate}
\end{corollary}

\begin{corollary}\label{coro:nets in these rings}
Let $P$ be a topological ring, $\Lambda$ an irreducible proset and $\{A^{(i)}\}_{i\in I}\subset \textup{M}_\Lambda(P)$ a net. The net is convergent if, and only if, for every $\alpha\in\Gamma(\Lambda)$ we have that $\{\pi_{\alpha}(A^{(i)})\}_{i\in I}$ is convergent, and the sequence congerges to $A$ defined coordinate-wise as follows:
$$A_{s_0,t_0}:=\left\{\begin{array}{ll}
     \lim_{i\in I}A^{(i)}_{s_0,t_0}, &  \mbox{if } s_0  \preceq t_0\in \Lambda \\
     0, &  \mbox{otherwise.}
 \end{array}
\right.$$
\end{corollary}

We can also easily describe a dense subset:\\
%\begin{proof}
%$1.$ To show they are closed note that $\textup{M}_\Lambda(J)=\bigcap_{\alpha\in \Gamma(\Lambda)}I^\Lambda_{\alpha}(P)+\textup{M}_\Lambda(J)$, that is, it is an intersection of clopen ideals hence it is closed. In a similar way, let $\{J_\sigma\}_{\sigma\in \Sigma}$ be a collection of two sided ideals of $P$ such that $\bigcap_{\sigma\in\Sigma}J_{\sigma}=\{0\}$ and note that $I^S_{S'}=\bigcap_{\sigma\in\Sigma} (I^S_{S'}+\textup{M}_\Lambda(J_\sigma))$.\\
%$2.$ and $3.$ follow trivially from the fact that two-sided ideals of a ring are open if, and only if, the quotient is a discrete ring and that $\textup{M}_{\bigsqcup_{i\in I}\Lambda_i}(P)$ is discrete if, and only if, $P$ is discrete and $\bigsqcup_{i\in I}\Lambda_i$ is finite.  
%\end{proof}
\textbf{Notation}:
Given $P$ a ring and $\Lambda$ a proset, we denote the following elements from $\textup{M}_\Lambda(P)$ as:
\begin{itemize}
    \item For $p\in P$ we define the element $p^{\Lambda}$ coordinatewise as: $$p^{\Lambda}_{s_1,s_2}:=\left\{\begin{array}{ll}
     p, &  \mbox{if } s_1=s_2\in \Lambda \\
     0, &  \mbox{otherwise.}
 \end{array}\right.$$
    \item For $S\subset \Lambda$ a subset we define the element $1^{S}$ coordinatewise as: $$1^{S}_{s_1,s_2}:=\left\{\begin{array}{ll}
     1, &  \mbox{if } s_1=s_2\in S \\
     0, &  \mbox{otherwise.}
 \end{array}\right.$$
    \item For $t_1  \preceq t_2\in \Lambda$ we define the element $e^{(t_1,t_2)}$ coordinatewise as: $$e^{(t_1,t_2)}_{s_1,s_2}:=\left\{\begin{array}{ll}
     1, &  \mbox{if } s_1=t_1 \text{ and }s_2=t_2 \\
     0, &  \mbox{otherwise.}
 \end{array}\right.$$ 
\end{itemize}

Notice that for every $s\in \Lambda$, by definition, $e^{(s,s)}=1^{\{s\}}$. Given $P$ a ring and $S\subset P$ a subset we will define $\langle S\rangle$ the minimal subring of $P$ containing $S$.

\begin{proposition}\label{prop:dense subset rings}\label{defi:dense subset}
Let $P$ be a topological ring, $S\subset P$ a dense subset of $P$ and $\Lambda$ a proset. Then $$\textup{M}_\Lambda(P)=\overline{\langle \{p^{\Lambda}\}_{p\in S}, \{1^{\{t\}}\}_{t\in \Lambda}, \{e^{(s_1,s_2)}\}_{s_1  \precnapprox_2\in\Lambda} \rangle}=\overline{\langle \{p^{\Lambda}\}_{p\in S}, \{e^{(s_1,s_2)}\}_{s_1  \preceq s_2\in\Lambda} \rangle}.$$
\end{proposition}

\begin{proof}
Define $R:= \langle \{p^{\Lambda}\}_{p\in S}, \{1^{\{t\}}\}_{t\in \Lambda}, \{e^{(s_1,s_2)}\}_{s_1  \preceq s_2\in\Lambda} \rangle$ and $\overline{R}$ the closure of $R$ in $\textup{M}_\Lambda(P)$. Because $S\subset P$ is dense it is clear that $\{p^{\Lambda}\}_{p\in P}\subset \overline{R}$, hence we can use these elements to write our convergent nets.\\
  This result is clearly true for $\Lambda$ finite, as the subring $R$ contains all the elements with finitely many non-zero coordinates. Assume $\Lambda$ is infinite and let $\Gamma(\Lambda)$ be as in Theorem \ref{Prop: prodiscrete ring}. For each $\alpha\in \Gamma(\Lambda)$ define 
$$A^{(\alpha)}_{s_1,s_2}:=\left\{\begin{array}{ll}
     A_{s_1,s_2}, &  \mbox{if } s_1, s_2\in\alpha \\
     0, &  \mbox{otherwise.}
 \end{array}\right.$$
 It is clear that these are all in $R$. By Corollary \ref{coro:nets in these rings} it follows that this net is convergent and it converges to $A$. Hence $\overline{R}=\textup{M}_{\Lambda}(P)$.
\end{proof}

\begin{corollary}
Let $\Lambda$ be a proset and $P$ a topological ring. Then $\textup{M}_\Lambda(P)$ is second countable if, and only if, $P$ is second countable and $\Lambda$ is countable.
\end{corollary}

\begin{proof}
If $P$ is second countable then there is $S\subset P$ a countable dense subset. By Proposition \ref{prop:dense subset rings} it follows that the set $\langle \{p^{\Lambda}\}_{p\in S}, \{1^{\{t\}}\}_{t\in \Lambda}, \{e^{(s_1,s_2)}\}_{s_1  \preceq s_2\in\Lambda} \rangle$ is dense in $\textup{M}_\Lambda(P)$. Because $S$ and $\Lambda$ are countable it follows that the dense subring $\langle \{p^{\Lambda}\}_{p\in S}, \{1^{\{t\}}\}_{t\in \Lambda}, \{e^{(s_1,s_2)}\}_{s_1  \preceq s_2\in\Lambda} \rangle$ is also countable, hence $\textup{M}_\Lambda(P)$ is also second countable.\\
  For the other side, notice that $\{A\in \textup{M}_\Lambda(P): \text{if } s_1\neq s_2\text{ then } A_{s_1,s_2}=0\}\cong\prod_{s\in \Lambda} P$ is a closed subring of $\textup{M}_\Lambda(P)$, and this subring is second countable if, and only if, $P$ is second countable and $\Lambda$ is countable.
\end{proof}

\section{Categorical properties}\label{Categorial properties}

\subsection{The functor \texorpdfstring{$\mathbf{M}_\_(P)$}{TEXT}}

In this section we will show that this can be extended to defining a contravariant functor from a category of partially ordered sets to the category of rings, showing that even products and inverse limits of these rings can be understood just by looking at limits and coproducts on the category or partially ordered sets. This question arise naturally from Proposition \ref{prop:disjoint union and product rings} and Theorem \ref{Prop: prodiscrete ring}, which points to the direction that coproducts are mapped to products and colimits to limits under such map.\\
  For this section we will need to restrict the maps we accept for the category of ordered sets, as not all maps would work. We then give arguments for why each condition is added.

\begin{definition}\label{defi:morphisms of prosets/posets}\index{order preserving map|ndx}\index{convex embedding|ndx}\index{FCC map|ndx}
Given $\Lambda, \Lambda'$ two prosets and $f:\Lambda\rightarrow\Lambda'$ a set function. We say that:
\begin{itemize}
    \item $f$ is \textbf{order preserving} if for $a\preceq b\in \Lambda$ then $f(a)\preceq f(b)\in \Lambda'$.
    \item $f$ is a \textbf{convex function} if it is an order preserving function and whenever $\alpha\subset\Lambda$ is a convex subset then $f(\alpha)$ is a convex subset of $\Lambda'$.
    \item $f$ is a \textbf{convex embedding} if it is a convex function and injective.
    \item $f$ is \textbf{constant}\index{constant|ndx} if $f(\Lambda)=\{t\}$, for some $t\in\Lambda'$
    \item Given that $\sqcup_{i\in I}\Lambda_i$ is the decomposition of $\Lambda$ into its components. we say $f$ \textbf{factors out into constant functions and convex embeddings (FCC)} if for each $i\in I$ either the function $f|_{\Lambda_i}:\Lambda_i\rightarrow\Lambda'$ is a constant function or a convex embedding.
\end{itemize}
\end{definition}
Let $f:\Lambda_1\rightarrow \Lambda_2$ be an order preserving morphism. A first idea for a map $\Tilde{f}:\textup{M}_{\Lambda_1}(P)\rightarrow \textup{M}_{\Lambda_2}(P)$ is as follows:
$$\begin{array}{llll}
    \tilde{f}:& \textup{M}_{\Lambda_2}(P) &\longrightarrow &  \textup{M}_{\Lambda_1}(P) \\
     & A & \longmapsto & \tilde{f}(A)
\end{array}$$
where $\tilde{f}(A)_{s_1,s_2}=A_{f(s_1),f(s_2)}$.

\begin{enumerate}
    \item If $f$ is injective, but not surjective, then $\Tilde{f}(1^{\Lambda_1})\neq 1^{\Lambda_2}$, hence it is not a ring homomorphism. For example:
$$f:2<\rightarrow 3<$$
$$f(0)=0,  f(1)=2$$
then
$$\Tilde{f}:\textup{M}_{2<}(P)\rightarrow \textup{M}_{3<}(P)$$
$$\left(\begin{array}{cc}
   a_{0,0}  &  a_{0,1} \\
    0 & a_{1,1}
\end{array}\right)\mapsto \left(\begin{array}{ccc}
   a_{0,0}  &  0 & a_{0,1}\\
    0 & 0 & 0\\
    0 & 0 & a_{1,1}
\end{array}\right).$$
\item If $f$ is surjective, but not injective, it does not preserve multiplication. For example:
$$f:2<\rightarrow 1$$
then
$$\Tilde{f}:\textup{M}_{2<}(P)\rightarrow \textup{M}_1(P)$$
$$\left(\begin{array}{cc}
   a_{0,0}  &  a_{0,1} \\
    0 & a_{1,1}
\end{array}\right)\mapsto a_{0,0}+a_{0,1}+a_{1,1}.$$
\end{enumerate}

\noindent Hence it is necessary to make some tweaks when defining a morphism between matrix rings in relation to a morphism between prosets.\\
  A second idea is reversing the direction of $f:\Lambda_1\rightarrow \Lambda_2$. We can then define $\mathbf{M}[f]$ on the set of generators (Proposition \ref{prop:dense subset rings}) of $\textup{M}_{\Lambda_2}(P)$ as follows:
$$\begin{array}{llll}
    \mathbf{M}[f]:&\textup{M}_{\Lambda_2}(P) &\longrightarrow & \textup{M}_{\Lambda_1}(P)\\
     & p^{\Lambda_2} & \longmapsto & p^{\Lambda_1}\\
     & e^{(s_1,s_2)}&\longmapsto & \left\{\begin{array}{ll}
                                          \sum_{t_1,t_2}e^{(t_1,t_2)}    &  \mbox{for }t_1\preceq t_2 \mbox{ and }f(t_1)=s_1, f(t_2)=s_2\\
                                       0      &  \mbox{otherwise}
                                    \end{array}\right.    
\end{array}$$

 In this case we have the following:

\begin{proposition}
Assume $\Lambda_1, \Lambda_2$ are prosets and $P$ is a ring. Let $f:\Lambda_1\rightarrow \Lambda_2$ be an order preserving morphism that is not convex. Then  $\mathbf{M}[f]$ is not well-defined.
\end{proposition}

\begin{proof}
Let $f$ be an order preserving that is not convex. If $f$ is not convex then there is $\alpha\in \Lambda_1$ a convex set such that $f(\alpha)$ is not convex in $\Lambda_2$. Hence either $f(\alpha)$ is not closed under intervals or it is not connected.\\
  If $f(\alpha)$ is not connected, there are $s_1, s_2\in \alpha$ such that $[s_1, s_2]\neq \emptyset$, but $[f(s_1), f(s_2)]=\emptyset$. That is, $s_1\preceq s_2$ but $f(s_1)\npreceq f(s_2)$, a contradiction to the assumption $f$ is order preserving.\\
  If $f$ is not closed under intervals, there are $s_1\preceq s_2\in \Lambda_1$ such that $[f(s_1), f(s_2)]$ is not in $f(\alpha)$. Then there is $f(s_1)\precnapprox t\precnapprox f(s_2)$ such that $f^{-1}(t)=\emptyset$. In this case we then have
$$\mathbf{M}[f](e^{(f(s_1),t)})\mathbf{M}[f](e^{(t,f(s_2))})=0$$
but 
$$\mathbf{M}[f](e^{(f(s_1),f(s_2))})$$
contains $e^{(s_1,s_2)}$ as a term, hence is non-zero. That is, $$\mathbf{M}[f](e^{(f(s_1),t)})\mathbf{M}[f](e^{(t,f(s_2))})\neq \mathbf{M}[f](e^{(f(s_1),f(s_2))}).$$

%  If it is not connected, there exist .\\
%  Now let $s_1,s_2$ as in the claim, and $t\in \Lambda_2$ be such that $f(s_1)< t\precnapprox f(s_2)$. Notice that $(e^{(s_1,s_2)}+e^{(s_1,t)}+e^{(t,s_2)})^2=e^{(s_1,s_2)}$. We will show the left side and right side of the equation cannot be the same given our initial condition.\\
%  On one hand
%\begin{equation}\label{eq:5}
%    \mathbf{M}[f](e^{(s_1,s_2)})=\sum_{t_1,t_2} e^{(t_1,t_2)},
%\end{equation}
%for $f(t_1)=f(s_1)$, $f(t_2)=f(s_2)$ and $t_1\precnapproxt_2$.\\
%  On the other hand $\mathbf{M}[f](e^{(f(s_1),f(s_2))}+e^{(f(s_1),t)}+e^{(t,f(s_2))})=\sum_{t_1,t_2} e^{(t_1,t_2)}$, for $f(t_1)=f(s_1)$, $f(t_2)=f(s_2)$ and $t_1\precnapproxt_2$. Hence 
%\begin{equation}\label{eq:6}
%    \mathbf{M}[f](e^{(f(s_1),f(s_2))}+e^{(f(s_1),t)}+e^{(t,f(s_2))})^2=(\sum_{t_1,t_2} e^{(t_1,t_2)})^2=\sum_{q_1,q_2}e^{(q_1,q_2)}
%\end{equation}
%such that $q_1\precnapproxq_2$, $f(q_1)=s_1$, $f(q_2)=s_2$ and there exists $q_1\precnapproxq\precnapproxq_2$ such that $f(q)\in\{s_1,s_2\}$. If there is no such $q_1,q_2$ then $(\ref{eq:6})=0$. If there exists such a $q$, fix a $q$ such that $|[q,q_2]|=2$. Then $e^{(q,q_2)}$ is part of the sum in $(\ref{eq:5})$ but not of the sum in $(\ref{eq:6})$. In both cases $(\ref{eq:5})\neq(\ref{eq:6})$.\\
%  Hence $\mathbf{M}[f]((e^{(s_1,s_2)}+e^{(s_1,t)}+e^{(t,s_2)})^2)\neq (\mathbf{M}[f](e^{(s_1,s_2)}+e^{(s_1,t)}+e^{(t,s_2)}))^2$ and $\mathbf{M}[f]$ cannot be a ring homomorphism.
\end{proof}

An example of the result above is as follows: let $f:\mathbf{2}<\longrightarrow \mathbf{3}<$ defined by $f(0)=0$, $f(1)=2$. This map is order preserving but not convex. The $\mathbf{M}[f]:{M}_{\mathbf{3}<}(P)\rightarrow {M}_{\mathbf{2}<}(P)$ would be defined as follows:
$$ \left(\begin{array}{lll}
a_{0,0} & a_{0,1} & a_{0,2} \\
0 & a_{1,1} & a_{1,2} \\
0 & 0 & a_{2,2} 
\end{array}\right) \mapsto \left(\begin{array}{ll}
a_{0,0} & a_{0,2} \\
0 & a_{2,2} 
\end{array}\right) $$
which is easily seen not to preserve multiplication. \\
  It is still not known if it the morphism being FCC is a necessary condition for $\mathbf{M}[f]$ to be well-defined, but there are some examples showing that being convex is not sufficient. For instance, let  $f:\mathbf{4<}\rightarrow\mathbf{2<}$ defined as $f(0)=f(1)=0$, $f(2)=f(3)=1$. The map $\mathbf{M}[f]:{M}_{\mathbf{2}<}(P)\rightarrow {M}_{\mathbf{4}<}(P)$ would be defined as follows:
$$ \left(\begin{array}{ll}
a_{0,0} & a_{0,1} \\
0 & a_{1,1} 
\end{array}\right) \mapsto \left(\begin{array}{llll}
a_{0,0} & 0 & a_{0,1} & a_{0,1} \\
0 & a_{0,0} & a_{0,1} & a_{0,1} \\
0 & 0 & a_{1,1} & 0\\
0 & 0 & 0 & a_{1,1}
\end{array}\right) $$
which is clearly not a ring homomorphism.
\begin{remark}\label{remark:fcc property}
Let $\Lambda^1$ be a proset and let $\sqcup_{i\in I}\Lambda_i$ the decomposition of $\Lambda^1$ into its components. Notice that if $f:\Lambda^1\rightarrow \Lambda^2$ is FCC then for every $s\in \Lambda^2$, $i\in I$ we have that either $f^{-1}(s)\cap \Lambda_i=\emptyset$, $f^{-1}(s)\cap \Lambda_i=\{t_i\}$ or $f^{-1}(s)\cap \Lambda_i=\Lambda_i$.
\end{remark}

%Notice that if $f:\Lambda_1\rightarrow \Lambda_2$ is a convex embedding, then the pre-image of convex subsets of $\Lambda_2$ is a convex subset of $\Lambda_1$.\  
\begin{definition}
We define the category category of locally finite prosets $\textbf{\textup{prosets}}_{\textup{fin}}$ as:
\begin{itemize}
    \item Objects: all locally finite prosets,
    \item Arrows: all the FCC morphisms between locally finite prosets.
\end{itemize}
\end{definition}

Notice that the category $\textbf{\textup{prosets}}_{\textup{fin}}$ has $0=\emptyset$ as its initial object and $1$ as its terminal object.\\
  Our goal now is to show that if $f$ is FCC then $\textbf{M}[f]$ can be extended to a ring homomorphism. Not only that, we can define a contravariant functor from the category $\textbf{prosets}_\text{fin}$ to $\textbf{TopRings}$.\\
  \indent Some remarks prior to the theorem:
  \begin{remark}\label{remark:relations among generators}\label{remark:order doesnt matter matrices}
Given $\Lambda$ a proset and $P$ a topological ring note that:
\begin{itemize}
    \item For every $s_1\leqslant s_2\in \Lambda$, $t_1\leqslant t_2\in \Lambda$
        $$e^{(s_1,s_2)}e^{(t_1,t_2)}=\left\{\begin{array}{ll}
                                        e^{(s_1,t_2)} &  \mbox{if }s_2=t_1 \\
                                        0 & \mbox{otherwise.}
                                    \end{array}\right.$$
    \item For every $s\in \Lambda$, 
    $$e^{(s,s)}=1^{\{s\}}.$$
    \item For every $S_1, S_2$ subsets of $\Lambda$                                  $$1^{S_1}1^{S_2}=\left\{\begin{array}{ll}
                                        1^{S_1\cap S_2} &  \mbox{if }S_1\cap S_2\neq\emptyset \\
                                        0 & \mbox{otherwise.}
                                    \end{array}\right.$$
                                    .
    \item For every $S\subset \Lambda$, $s_1\leqslant s_2 \in\Lambda$
        $$1^{S}e^{(s_1,s_2)}=\left\{\begin{array}{ll}
                                        e^{(s_1,s_2)} &  \mbox{if }s_1\in S \\
                                        0 & \mbox{otherwise.}
                                    \end{array}\right.$$
        $$e^{(s_1,s_2)}1^{S}=\left\{\begin{array}{ll}
                                        e^{(s_1,s_2)} &  \mbox{if }s_2\in S \\
                                        0 & \mbox{otherwise.}
                                    \end{array}\right.$$
    \item For every $S\subset \Lambda$, $p\in P$
    $$p^{\Lambda}1^{S}=1^{S}p^{\Lambda}.$$ 
    \item For every $s_1\leqslant s_2\in \Lambda$, $p\in P$
    $$p^{\Lambda}e^{(s_1,s_2)}=e^{(s_1,s_2)}p^{\Lambda}.$$ 
    \item Given $s,s_1,s_2\in \Lambda$ and $A\in \textup{M}_\Lambda(P)$, then
 $$(1^{\{s\}}A)_{t_1,t_2}=\left\{\begin{array}{ll}
     A_{t_1,t_2}, &  \mbox{if } t_1=s  \\
     0, &  \mbox{otherwise.}
 \end{array}\right. $$
 $$(A1^{\{s\}})_{t_1,t_2}=\left\{\begin{array}{ll}
     A_{t_1,t_2}, &  \mbox{if }t_2=s \\
     0, &  \mbox{otherwise.}
 \end{array}\right.$$
$$(1^{\{s_1\}}A1^{\{s_2\}})_{t_1,t_2}=\left\{\begin{array}{ll}
     A_{t_1,t_2}, &  \mbox{if } t_1=s_1 \text{ and } t_2=s_2 \\
     0, &  \mbox{otherwise.}
 \end{array}\right. $$
 $$(A1^{\{s\}}A)_{t_1,t_2}=\left\{\begin{array}{ll}
     A_{t_1,s}A_{s,t_2}, &  \mbox{if }s\in [t_1, t_2] \\
     0, &  \mbox{otherwise.}
 \end{array}\right.$$
\end{itemize}
These all can be proven by expanding the sums in each coordinate $(t_1, t_2)$.
\end{remark}

\begin{theorem}\label{thrm:functor construction} \index{$\textbf{\textup{M}}_{\_}(P)$|ndx}\index{contravariant functor|ndx}
Let $P$ be a topological ring. The map 
$$\begin{array}{cccc}
 \textbf{\textup{M}}_{\_}(P):    & \textbf{\textup{prosets}}_{\textup{fin}}&\Longrightarrow & \textbf{\textup{TopRings}} \\
& \Lambda &\longmapsto     & \textup{M}_{\Lambda}(P)\\
& f& \longmapsto &\textbf{\textup{M}}[f]
\end{array}$$
is a contravariant functor.
\end{theorem}

\begin{proof}
Given $P$ a topological ring, the functor $\mathbf{M}_{\_}(P)$ maps the object $\Lambda$ from $posets$ to $\mathbf{M}_{\_}(P)(\Lambda):=\textup{M}_{\Lambda}(P)$, where $\textup{M}_\emptyset(P):=\{0\}$ is defined to be the trivial ring. We now need to show the functor map FCC maps to continuous ring homomorphims. To simplify the notation we will denote $\mathbf{M}_{\_}(P)(f):=\mathbf{M}[f]$, for $f$ an FCC map. The proof will follow by defining a map from the generators of the domain and showing it extends to a continuous ring homomorphism. Remark \ref{remark:relations among generators} and Proposition \ref{prop:dense subset rings} will be central for this part of the proof. It will then remain to show that composition of functions is mapped to composition of functions.\\
  If $\Lambda_1$ is the empty poset and $f:\emptyset\rightarrow \Lambda_2$ is the unique morphism from the empty set to $\Lambda_2$, define $\mathbf{M}[f]$ as the unique morphism from $\textup{M}_{\Lambda_2}(P)$ to $\textup{M}_\emptyset(P)$.\\
  For $\Lambda_2$ a non-empty set, we extend $\mathbf{M}[f]$ for some finite sums as follows:
\begin{itemize}
    \item If $A,B\in \{p^{\Lambda}\}_{p\in P}\cup \{e^{(s_1,s_2)}\}_{s_1  \preceq s_2\in\Lambda_2}$, the generating set, we define $$\mathbf{M}[f](A+B):=\mathbf{M}[f](A)+\mathbf{M}[f](B).$$
    \item If $A\in \{e^{(s_1,s_2)}\}_{s_1  \preceq s_2\in\Lambda_2}$ and $p\in P$, we define 
    $$p^{\Lambda_1}\mathbf{M}[f](A)=\mathbf{M}[f](p^{\Lambda_2}A)=\mathbf{M}[f](Ap^{\Lambda_2})=\mathbf{M}[f](A)p^{\Lambda_1}.$$
\end{itemize}

We want now to show that such a map can be extended to a continuous ring homomorphism. First, we will show that such a map can be extended to a ring homomorphism on the dense subset from Proposition \ref{prop:dense subset rings} and then extend it continuously to the whole ring. For that let $\Lambda_1=\sqcup_{i\in I}\Lambda_i$ the decomposition of $\Lambda_1$ into its irreducible components.

  \textit{Claim 1}: Let $t_1\precnapprox t_2\in \Lambda_1$ be such that $f(t_1)=s_1$ and $f(t_2)=s_2$ for some $s_1\precnapprox s_2\in\Lambda_2$. If $s_1\precnapprox s_2\precnapprox s_3$, for some $s_3\in \Lambda_2$ then one of the following happen:
\begin{itemize}
    \item $e^{(t_1,t_2)}\mathbf{M}[f](e^{(s_2,s_3)})=0$, or
    \item $e^{(t_1,t_2)}\mathbf{M}[f](e^{(s_2,s_3)})=e^{(t_1,t_3)}$ for a unique $t_3\in\Lambda_1$ such that $t_1\precnapprox t_2\precnapprox t_3$ and $f(t_3)=s_3$.
\end{itemize}
Assume $e^{(t_1,t_2)}\mathbf{M}[f](e^{(s_2,s_3)})\neq 0$. Then expanding it, we have the right-hand side of
$$e^{(t_1,t_2)}\mathbf{M}[f](e^{(s_2,s_3)})=e^{(t_1,t_2)}\left(\sum e^{(t_2',t_3')} \right)$$
has at least one pair $(t_2', t_3')$ such that $f(t_2')=s_2$, $t_2'=t_2$, and $f(t_3')=s_3$ (Remark \ref{remark:relations among generators}). Hence, restricted to the component containing $t_1, t_2$ the function $f$ is a convex embedding. By Remark \ref{remark:fcc property}, such the pair $(t_2', t_3')$ is unique, as it has to be on the same component as $t_1, t_2$. That is, if $e^{(t_1,t_2)}\mathbf{M}[f](e^{(s_2,s_3)})\neq 0$ then $e^{(t_1,t_2)}\mathbf{M}[f](e^{(s_2,s_3)})=e^{(t_1,t_3')}$.\\
  \textit{Claim 2}: Let $s_1\precnapprox s_2$ be elements in $\Lambda_2$. Then $\mathbf{M}[f](e^{(s_1,s_2)})\mathbf{M}[f](e^{(t_1,t_2)})=\mathbf{M}[f](e^{(s_1,s_2)}e^{(t_1,t_2)})$ for every $t_1, t_2\in \Lambda_2$.\\
  Assume $s_1, s_2, t_1, t_2\in\Lambda_2$ are such that $s_1\precnapprox s_2$, $s_2=t_1$ and $t_1\precnapprox t_2$, as otherwise the result would be trivially $0$ on both sides. Then $e^{(s_1,s_2)}e^{(t_1,t_2)}=e^{(s_1,t_2)}$ and, as $f$ is FCC, by Remark \ref{remark:fcc property} the right-hand side is
\begin{equation}\label{equation:left_hand}
\mathbf{M}[f](e^{(s_1,t_2)})=\sum_{j\in J}e^{(q_{j_1},q_{j_2})}    
\end{equation}
for $J$ a subset of $I$, $q_{j_1}\precnapprox q_{j_2}\in\Lambda_j$ and $f^{-1}(s_1)\cap \Lambda_j=\{q_{j_1}\}$, $f^{-1}(t_2)\cap \Lambda_j=\{q_{j_2}\}$. For the left-hand side, 
\begin{equation}\label{equation:right_hand}
\mathbf{M}[f](e^{(s_1,s_2)})\mathbf{M}[f](e^{(t_1,t_2)})=\left(\sum_{k\in K}e^{(s_{k_1},s_{k_2})}\right)\left(\sum_{l\in L}e^{(t_{l_1},t_{l_2})}\right)    
\end{equation}
where $K$ and $L$ are subsets of $I$, for $k\in K$ we have $s_{k_1}\precnapprox s_{k_2}\in\Lambda_k$, $f^{-1}(s_1)\cap \Lambda_k=\{s_{k_1}\}$, $f^{-1}(s_2)\cap \Lambda_k=\{s_{k_2}\}$ and for $l\in L$ we have $t_{l_1}\precnapprox t_{l_2}\in\Lambda_l$, $f^{-1}(t_1)\cap \Lambda_l=\{t_{l_1}\}$, $f^{-1}(t_2)\cap \Lambda_l=\{t_{l_2}\}$. As we assumed $s_1\precnapprox s_2=t_1\precnapprox t_2$ we have that
$$\mathbf{M}[f](e^{(s_1,s_2)})\mathbf{M}[f](e^{(t_1,t_2)})=\left(\sum_{p\in K\cap L}e^{(s_{p_1},t_{p_2})}\right).$$
We need to show $K\cap L=J$, and the result will follow from Remark \ref{remark:fcc property}. If $j\in J$ then there are $q_{j_1}\precnapprox q_{j_2}$ such that $f(q_{j_1})=s_1$, $f(q_{j_2})=t_2$, hence $f|_{\Lambda_j}$ is convex embedding. Then there is $q\in[q_{j_1},q_{j_2}]$ such that $f(q)=s_2=t_1$, that is, $e^{(q_{j_1},q)}$ and $e^{(q,q_{j_2})}$ show up as terms in the right hand side of \ref{equation:right_hand}. Hence $j\in K\cap L$.\\
  On the other hand, if $p\in K\cap L$, then there are $s_{p_1}\precnapprox t\precnapprox t_{p_2}\in \Lambda_p$ such that $f(s_{p_1})=s_1$, $f(t)=s_2=t_1$, $f(t_{p_2})=t_2$. Hence the term $e^{(s_1,t_2)}$ appears on the right-hand side of Equation \ref{equation:left_hand}. That is, $p\in J$.\\
  \textit{Claim 3}: $\mathbf{M}[f]$ is well-defined and continuous.\\
  Claim $1.$ and $2.$ implies it is well-defined for all elements that are finite sums and finite products of the generators. Hence it is well-defined for a dense subset of $\textup{M}_{\Lambda_2}(P)$ (Corollary \ref{prop:dense subset rings}). Given $\{A_\alpha\}_{\alpha\in \Gamma(\Lambda_2)}\subset \langle \{p^{\Lambda_2}\}_{p\in S}, \{e^{(s_1,s_2)}\}_{s_1  \preceq s_2\in\Lambda_2} \rangle $ a convergent net, Corollary \ref{coro:nets in these rings} then implies
$$\mathbf{M}[f]\left(\lim_{\alpha\in \Gamma(\Lambda_2)}(A_\alpha)\right)=\lim_{\alpha\in \Gamma(\Lambda_2)}\mathbf{M}[f](A_\alpha).$$
Hence the morphism can be extended to a continuous ring homomorphism.\\
  \textit{Claim 4}: Given $f_1:\Lambda_1\rightarrow \Lambda_2$ and $f_2:\Lambda_2\rightarrow\Lambda_3$ then $\mathbf{M}[f_2\circ f_1]=\mathbf{M}[f_1]\circ\mathbf{M}[f_2]$.\\
  The result is trivial when $\Lambda_1=\emptyset$. Assume $\Lambda_1\neq \emptyset$. The proof will be given by looking at where the generators of $\textup{M}_{\Lambda_3}$ are mapped.\\
  It is easy to see that for $p\in P$
$$\textbf{M}[f_2\circ f_1](p^{(\Lambda_3})=p^{\Lambda_1}=\textbf{M}[f_1]\circ \textbf{M}[f_2](p^{\Lambda_3}).$$
  For the $e^{(s_1,s_2)}$ case, notice that if $\mathbf{M}[f_2\circ f_1](e^{(s_1,s_2)})\neq 0$ then:\\
\begin{equation}\label{eq:3}
    \begin{array}{lll}
    \mathbf{M}[f_2\circ f_1](e^{(s_1,s_2)})& = \sum_{t_1,t_2}e^{(t_1,t_2)}    &  \mbox{for }t_1\preceq t_2\in\Lambda_1, \mbox{and }f_2\circ f_1(t_1)=s_1,\\
    &  &  f_2\circ f_1(t_2)=s_2
\end{array}
\end{equation}
and 
\begin{equation}\label{eq:4}
\begin{array}{lll}
  \mathbf{M}[f_1]\circ\mathbf{M}[f_2](e^{(s_1,s_2)})&  =\mathbf{M}[f_1](\sum_{q_1,q_2}e^{(q_1,q_2)})    &  \mbox{where }q_1\preceq q_2\in\Lambda_2 \mbox{ and }f_2(q_1)=s_1, \\
    & &    f_2(q_2)=s_2\\
  &   = \sum_{q_1,q_2}(\sum_{t_1,t_2}e^{(t_1,t_2)})&\mbox{for }t_1\preceq t_2 \in \Lambda_1, \mbox{ and }f_2(q_1)=s_1, \\
  & &   f_2(q_2)=s_2, f_1(t_1)=q_1, f_1(t_2)=q_2.
\end{array}    
\end{equation}
Claim 1 implies that the right-hand side of Equation (\ref{eq:4}) has to be the same as the right-hand side of Equation (\ref{eq:3}). Then $\mathbf{M}[f_1]\circ\mathbf{M}[f_2](e^{(s_1,s_2)})=\mathbf{M}[f_2\circ f_1](e^{(s_1,s_2)})$.\\
  It then follows that $\mathbf{M}_{\_}(P)$ is a contravariant functor.

\end{proof}

We call the functor $\mathbf{M}_{\_}(P)$ the \textbf{incidence functor}\index{incidence funtor|ndx}.

\begin{remark}\label{remark:surj mapped to inj}
Notice that given $f:\Lambda^1\rightarrow \Lambda^2$ an FCC map and two elements $e^{(s_1,s_2)}, e^{(t_1,t_2)}$ on the generator set of $M_{\Lambda^2}(P)$, if they are distinct and have no trivial image then the image under $\mathbf{M}[f]$ of both elements are disjoint. Hence if $f$ is a surjective FCC map we have that $\mathbf{M}[f]$ is injective.
\end{remark}

%Now that we have the functor, it is interesting to also understand why the categories were defined in the way they were. We now will make some comments on why the conditions of $f$ is FCC is necessary for the functions $\mathbf{M}[f]$ as defined in Theorem \ref{thrm:functor construction} be well-defined.

\subsection{Sending colimits to limits} 

In this section we will work with properties of the category $posets$ that later can be used to show our contravariant functor maps colimits to limits.

\begin{lemma}\label{lemma:existence coproducts posets}\index{$\textbf{\textup{prosets}}_{\textup{fin}}$|ndx}
Given $\{\Lambda_i\}_{i\in I}$ be an arbitrary collection of objects in $\textbf{\textup{prosets}}_{\textup{fin}}$, their coproduct exist and it is $\sqcup_{i\in I}\Lambda_i$, as defined in Definition \ref{defi:disjoint union prosets}.
\end{lemma}

\begin{proof}
Let $\{\Lambda_i\}_{i\in I}$  as in the statement and define $\Lambda:=\sqcup_{i\in I}\Lambda_i$. For each $i\in I$ let $i_i:\Lambda_i\rightarrow\Lambda$ be the embeddings sending $\Lambda_i$ to its copy in $\Lambda$.\\
  Claim: $\Lambda$ is the coproduct of  $\{\Lambda_i\}_{i\in I}$.\\
  Let $\Lambda'$ be another locally finite proset and, for each $i\in I$, let $f_i:\Lambda_i\rightarrow \Lambda'$ a FCC map. Define $(f_i)_{i\in I}:\Lambda\rightarrow \Lambda'$ by $(f_i)_{i\in I}(s_i)=f_i(s_i)$ if $s_i\in \Lambda_i$. It is clear that $(f_i)_{i\in I}\circ i_i=f_i$ and $(f_i)_{i\in I}$ is the unique FCC morphism satisfying this property. Hence $\Lambda$ is the coproduct of the collection $\{\Lambda_i\}_{i\in I}$.

$$% https://tikzcd.yichuanshen.de/#N4Igdg9gJgpgziAXAbVABwnAlgFyxMJZABgBoAmAXVJADcBDAGwFcYkQAdDgGXoFsARlHoB9LCAC+pdJlz5CKMsWp0mrdl16Dhk6SAzY8BIuQoqGLNok49+Q+gHJJKmFADm8IqABmAJwh8SACMNDgQSKYgjPQCMIwACrJGCiBYYNiwIDQW6tbeWVFpViDCcAAWrro+-oGIZCBhwTTRsQlJ8uxpGWzZasVYYlUgfgFI9Y2IkS1xiYYd1l1Ymb2W7N6DEpQSQA
\begin{tikzcd}
\Lambda \arrow[rrdd, "(f_i)_{i\in I}" description, dashed]                         &  &          \\
                                                                      &  &          \\
\Lambda_i \arrow[uu, "i_i" description] \arrow[rr, "f_i" description] &  & \Lambda'
\end{tikzcd}$$
\end{proof}

The following follows directly from Lemma \ref{lemma:existence coproducts posets} and Proposition \ref{prop:disjoint union and product rings}.

\begin{corollary}\label{coro:functor maps coproducts to products}
The incidence functor $\mathbf{M}_\_(P)$ maps coproducts to products.
\end{corollary}

\begin{definition}[Quotient order]
Let $\Lambda=\sqcup_{i\in I}\Lambda_i$ be a poset with order given by the disjoint union. We will say $\equiv\subset \Lambda\times\Lambda$ is a \textbf{FCC equivalence relation} if there is a FCC surjective map $f:\Lambda\rightarrow \Lambda'$ such that given $s_1,s_2\in \Lambda$ then $s_1\equiv s_2$ if, and only if, $f(s_1)=f(s_2)$. We write $\Lambda':=\Lambda/\equiv$ and call this set the \textbf{FCC quotient} under $\equiv$.
%given $s_1\sim s_2$ then or $s_1=s_2$, or if $s_1\neq s_2$ and exists $i\in I$ such that $s_1,s_2\in \Lambda_i$ then for every $s\in \Lambda_i$ we have $s\sim s_1$, or there exists $i_1\neq i_2\in I$ such that $s_1\in \Lambda_{i_1}$ and $s_2\in \Lambda_{i_2}$. The \textbf{quotient} will be denoted $\Lambda/\sim$ with elements the equivalence classes under this relation. For $[t_1],[t_2]\in \Lambda/\sim$, we define the quotient order as $[t_1]\leqslant [t_2]$ if there exists $q_1\in [t_1]$, $q_2\in [t_2]$ and $s_1,s_2\in \Lambda$ such that $q_1\leqslant s_1\sim s_2 \leqslant q_2$.
\end{definition}

For the next result we will prove the existence of pushouts in the category $\textbf{\textup{prosets}}_{\textup{fin}}$. We will first give a remark on how such construction works in the case the proset $\Lambda$ is irreducible, and the proof will then be given for the general case.

\begin{remark}
Let $\Lambda$, $\Lambda_1$ and $\Lambda_2$ be prosets and $f:\Lambda\rightarrow \Lambda_1$, $g:\Lambda\rightarrow \Lambda_2$ be FCC maps. We want to build the pushout diagram:
$$% https://tikzcd.yichuanshen.de/#N4Igdg9gJgpgziAXAbVABwnAlgFyxMJZABgBpiBdUkANwEMAbAVxiRAB12AZOgWwCModEAF9S6TLnyEUZAIxVajFm048BQgPoAmUeJAZseAkTnlF9Zq0QdufQXU1y9Eo9NOkF1Syptr7WsAAZqQA5iKiijBQofBEoEEAThC8SGQgOBBIZiAMdPwwDAAKksYyIFhg2LAg3srWIKEuIEkpadSZSNrUeQXFpe42ldWsdVZsQc2tqYjdGVmIAMw9+YUlbiZDVVg1Y74GTlPJMzmdS3sNaDq1uav9G+XDO6wiFCJAA
\begin{tikzcd}
\Lambda \arrow[d, "g" description] \arrow[r, "f" description] & \Lambda_1 \arrow[d, "p_1" description] \\
\Lambda_2 \arrow[r, "p_2" description]                        & {\Lambda_{f,g}}                       
\end{tikzcd}$$
For the case $\Lambda=\emptyset$, as the empty set is the initial object of $\textbf{\textup{prosets}}_{\textup{fin}}$,  $\Lambda_{f,g}=\Lambda_1\sqcup \Lambda_2$ and $p_1$, $p_2$ are the embeddings of $\Lambda_1, \Lambda_2$ into $\Lambda_1\sqcup \Lambda_2$.\\
  Assume $\Lambda$ is irreducible and non-empty. As $\Lambda$ is irreducible, there exists an unique $j_0\in J$ and $k_0\in K$ such that $f(\Lambda)\subset \Lambda_j$, $g(\Lambda)\subset \Lambda_k$. The desired proset $\Lambda_{f,g}$ will be as follows: 
$$\Lambda_{f,g}=\left(\bigsqcup_{j\in J\backslash\{j_0\}}\Lambda_j \right)\sqcup\left(\bigsqcup_{k\in K\backslash\{k_0\}}\Lambda_k \right)\sqcup \Lambda_{\{j_0,k_0\}}.$$
The proset $\Lambda_{\{j_0,k_0\}}$ is the FCC quootient $\Lambda_{j_0}\sqcup \Lambda_{k_0}$ with equivalence relation given by:
\begin{itemize}
    \item for every $t\in \Lambda$, $f(t)\equiv g(t)$,
    \item if $\Lambda\neq 1$, $f$ is a convex embedding and $g$ is a constant function then for every $s_1, s_2\in \Lambda_{j_0}$ we have $s_1\equiv s_2$,
    \item if $\Lambda\neq 1$, $g$ is a convex embedding and $f$ is a constant function then for every $s_1, s_2\in \Lambda_{k_0}$ we have $s_1\equiv s_2$.
\end{itemize}
This is the maximal FCC quotient of $\Lambda_{j_0}\sqcup \Lambda_{k_0}$ such that $\pi\circ i_{j_0}\circ f= \pi \circ i_{k_0}\circ g$. Here the maps
$\pi:\Lambda_{j_0}\sqcup \Lambda_{k_0}\rightarrow \Lambda_{\{j_0,k_0\}}$ is given by the FCC quotient under $\equiv$, and $i_{j_0}$, $i_{k_0}$ are the embeddings from $\Lambda_{j_0}$, $\Lambda_{k_0}$ into their respective components in the disjoint union $\Lambda_{j_0}\sqcup\Lambda_{k_0}$. It is then the case 
\begin{itemize}
    \item $p_1|_{\Lambda_{j_0}}=\pi\circ i_{j_0}$,
    \item for $j\neq j_0$, $p_1|_{\Lambda_j}=i_j$,
    \item $p_2|_{\Lambda_{k_0}}=\pi \circ i_{k_0}$,
    \item and for $k\neq k_0$, $p_2|_{\Lambda_k}=i_k$.
\end{itemize}
As $p_1$, $p_2$ is defined on all the components of $\Lambda_{\{f,g\}}$, it is uniquely defined for the whole proset. Notice that $\Lambda_{f,g}$ is an FCC quotient of $\Lambda_1\sqcup \Lambda_2$.
\end{remark}

\begin{lemma}\label{lemma:existence of pushouts}
Let $\Lambda,\Lambda_1,\Lambda_2$ be locally finite prosets and $f:\Lambda\rightarrow \Lambda_1$, $g:\Lambda\rightarrow \Lambda_2$ FCC maps. Then there exists $\Lambda_{f,g}$ a locally finite proset and $p_1:\Lambda_1\rightarrow \Lambda_{f,g}$, $p_2:\Lambda_2\rightarrow \Lambda_{f,g}$ FCC maps such that the diagram
$$% https://tikzcd.yichuanshen.de/#N4Igdg9gJgpgziAXAbVABwnAlgFyxMJZABgBpiBdUkANwEMAbAVxiRAB12AZOgWwCModEAF9S6TLnyEUZAIxVajFm048BQgPoAmUeJAZseAkTnlF9Zq0QdufQXU1y9Eo9NOkF1Syptr7WsAAZqQA5iKiijBQofBEoEEAThC8SGQgOBBIZiAMdPwwDAAKksYyIFhg2LAg3srWIKEuIEkpadSZSNrUeQXFpe42ldWsdVZsQc2tqYjdGVmIAMw9+YUlbiZDVVg1Y74GTlPJMzmdS3sNaDq1uav9G+XDO6wiFCJAA
\begin{tikzcd}
\Lambda \arrow[d, "g" description] \arrow[r, "f" description] & \Lambda_1 \arrow[d, "p_1" description] \\
\Lambda_2 \arrow[r, "p_2" description]                        & {\Lambda_{f,g}}                       
\end{tikzcd}$$
is a pushout diagram.
\end{lemma}

\begin{proof}
For the proof we will write $\Lambda=\sqcup_{i\in I}\Lambda_i$, $\Lambda_1=\sqcup_{j\in J}\Lambda_j$ and $\Lambda_2=\sqcup_{k\in K}\Lambda_k$ the decomposition of these sets into their components.\\
 
%For each $i\in I$ define $f_i:=f|_{\Lambda_i}$, $g_i:=g|_{\Lambda_i}$, $j_i\in J$ the element such that $f_i(\Lambda_i)\subset \Lambda_{j_i}$ and $k_i\in K$ the element such that $g_i(\Lambda_i)\subset \Lambda_{k_i}$. 
Define the sets:
$$\begin{array}{l}
    J_0=\{j\in J:  \text{for all }i\in I, f(\Lambda_i)\cap \Lambda_j=\emptyset\}\subset J\\
    J_1=\{j\in J\backslash J_0:  \forall i\in I, \text{ if }f(\Lambda_i)\subset \Lambda_j \text{ then }f|_{\Lambda_i} \text{ is a convex embedding}\}\subset J \\
    J_2=\{j\in J\backslash J_0:  \forall i\in I, \text{ if }\Lambda_i\neq 1, f(\Lambda_i)\subset \Lambda_j \text{ then }f|_{\Lambda_i} \text{ is a constant function}\}\subset J \\ 
    J_3=J\backslash (J_0\cup J_1\cup J_2) \\
    K_0=\{k\in K:  \text{for all }i\in I, g(\Lambda_i)\cap \Lambda_k=\emptyset\}\subset K\\
    K_1=\{k\in K\backslash K_0:  \forall i\in I, \text{ if }g(\Lambda_i)\subset \Lambda_k \text{ then }g|_{\Lambda_i} \text{ is a convex embedding}\}\subset K  \\
    K_2=\{k\in K\backslash K_0:  \forall i\in I, \text{ if }\Lambda_i\neq 1, g(\Lambda_i)\subset \Lambda_k \text{ then }g|_{\Lambda_i} \text{ is a constant function}\}\subset K  \\
    K_3= K\backslash (K_0\cup K_1 \cup K_2).
\end{array}$$
For each $l\in J\sqcup L$ also define the set $\langle l\rangle$ as the minimal set such that:
\begin{itemize}
    \item $l\in \langle l\rangle$.
    \item Let $k\in K$ be such that $k\in \langle l\rangle$. If $i\in I$ is such that $g(\Lambda_i)\subset \Lambda_k$ and $f(\Lambda_i)\subset \Lambda_j$, then $j\in \langle l\rangle$. 
    \item Let $j\in J$ be such that $j\in \langle l\rangle$. If $i\in I$ is such $f(\Lambda_i)\subset \Lambda_j$ and $g(\Lambda_i)\subset \Lambda_k$, then $k\in \langle l\rangle$. 
\end{itemize}
In other words, $\langle l\rangle$ is the set of components that contains $\Lambda_l$ and that will be ``glued together'' under the pushout. Notice that if $l\in J_0\sqcup K_0$ then $\langle l\rangle=\{l\}$. \\
  Let $L=\{\langle l\rangle:  l\in J\sqcup K\}$ the partition of $J\sqcup K$ defined as above. Define $\Lambda_{f,g}$ as follows:
$$\Lambda_{f,g}:=\bigsqcup_{\langle l\rangle\in L}\Lambda_{\langle l\rangle}$$
where $\Lambda_{\langle l\rangle}$ is defined as
$$\Lambda_{\langle l\rangle}=\left.\left(\bigsqcup_{l\in \langle l\rangle}\Lambda_l\right)\right/\equiv_{l}$$
where
\begin{itemize}
    \item for every $s\in \Lambda$, if there is $l\in \langle l\rangle$ such $f(s)\in \Lambda_l$ then $f(s)\equiv_{l} g(s)$.
    \item if $l\in \langle l\rangle$ is such that $l\in J_3\cup K_3$ then for every $s_1, s_2\in \Lambda_l$ we have $s_1\equiv_{l} s_2$.
    \item if $j\in J_1$ and $k\in K_2$ are such that there is $i\in I$ with $f(\Lambda_i)\subset \Lambda_j$, $g(\Lambda_i)\subset \Lambda_k$ then for every $s_1, s_2\in \Lambda_j$  we have $s_1\equiv_{l} s_2$.
    \item if $j\in J_2$ and $k\in K_1$ are such that there is $i\in I$ with $f(\Lambda_i)\subset \Lambda_j$, $g(\Lambda_i)\subset \Lambda_k$ then for every $s_1, s_2\in \Lambda_k$ we have $s_1\equiv_{l} s_2$.
\end{itemize}
with the order given by the FCC quotient. Notice that the construction is similar to the irreducible case, but here we have more components being ``glued together''. Notice that $\Lambda_{f,g}=(\Lambda_1\sqcup \Lambda_2)/\equiv$, for the $\equiv$ FCC equivalence relation defined above on each component $\Lambda_{\langle l\rangle}$. Let $i_1:\Lambda_1 \rightarrow\Lambda_1\sqcup \Lambda_2$ and $i_2:\Lambda_2\rightarrow \Lambda_1\sqcup \Lambda_2$ the embeddings from the coproduct and $\pi:\Lambda_1\sqcup \Lambda_2\rightarrow \Lambda_{f,g}$ given by the FCC quotient. We then define $p_1:=\pi\circ i_1$ and $p_2:=\pi\circ i_2$.\\
  Notice that $\Lambda_{f,g}$ is defined so that it is the maximal FCC quotient of $\Lambda_1\sqcup \Lambda_2$ such that there exists an unique FCC maps $p_1$, $p_2$ with $p_1\circ f=p_2\circ g$. The construction implies that if there are FCC maps $f':\Lambda_1\rightarrow \Lambda'$ and $g':\Lambda_2\rightarrow \Lambda'$, then there exists a unique FCC map $p':\Lambda_{f,g}\rightarrow \Lambda'$ such that $f'=p'\circ p_1$ and $g'=p'\circ p_2$. 
$$% https://tikzcd.yichuanshen.de/#N4Igdg9gJgpgziAXAbVABwnAlgFyxMJZABgBpiBdUkANwEMAbAVxiRAB12AZOgWwCModEAF9S6TLnyEUZAIxVajFm048BQgPoAmUeJAZseAkTnlF9Zq0QdufQXU1y9Eo9NOkF1Syptr7WsAAZqQA5iIuBpLGMsjapNoWyta26g4A5KKKMFCh8ESgQQBOELxIZCA4EEhmIAx0-DAMAArR7jZYYNiwIN7JbKGRxaXl1FVI8XUNTa1uJh1dWD19VmxBQyVliJPjiADM1PWNLW3zIJ3drCu+Bk4bI4i1uwdKqzZoOr1Tx7NSZxdLVhiQqbGpjaqIAAsh2mJzmMnOi2WrxuoUywJAwy2Owh0O+M1OCIByJ8KSC6P0WKQL12eKOBPhbGJVzqnRSQjgAAscl9SWw0FkREA
\begin{tikzcd}
\Lambda \arrow[d, "g" description] \arrow[r, "f" description]        & \Lambda_1 \arrow[d, "p_1" description] \arrow[rdd, "f'" description] &          \\
\Lambda_2 \arrow[r, "p_2" description] \arrow[rrd, "g'" description] & {\Lambda_{f,g}} \arrow[rd, "p'" description, dashed]                  &          \\
                                                                     &                                                                      & \Lambda'
\end{tikzcd}$$
\end{proof}

%The following isomorphism theorem for FCC maps follows trivially from definition.

%\begin{lemma}\label{lemma:image as quotient}
%Let $f:\Lambda_1\rightarrow\Lambda_2$ be an FCC embedding. Then $Im(f)\cong \Lambda_1/\sim$ for some FCC relation.
%\end{lemma}
Note that the intuitive idea of pushouts of prosets is getting two prosets, $\Lambda_1$ and $\Lambda_2$, with some common convex subsets, given by the maps $f$ and $g$ on each component of $\Lambda$, and making these subsets an intersection of them. Below we have an example of that:

\begin{example}
Let $\alpha=\mathbf{2}<=\{0,1\}$, $\Lambda_1=\mathbf{Z}$ and $\Lambda_2=\mathbf{Zig}$ and define 
$$\begin{array}{llll}
    i_1: & \mathbf{2}< & \longrightarrow & \mathbf{Z}\\
     & 0 & \longmapsto & 0\\
    & 1 & \longmapsto & -1,
\end{array}
$$
and
$$\begin{array}{llll}
    i_2: & \mathbf{2}< & \longrightarrow & \mathbf{Zig}\\
     & 0 & \longmapsto & 0\\
    & 1 & \longmapsto & -1.
\end{array}$$
The pushout of this diagram is the following poset:
$$%https://q.uiver.app/#q=WzAsMTIsWzUsMCwiXFxkb3RzIl0sWzQsMSwiMSJdLFsxLDIsIi0yIl0sWzMsMiwiMCJdLFs1LDIsIjIiXSxbMCwzLCJcXGxkb3RzIl0sWzIsMywiLTEiXSxbNCwzLCIxIl0sWzYsMywiXFxsZG90cyJdLFsxLDQsIi0yIl0sWzAsNSwiXFxkb3RzIl0sWzUsM10sWzAsMSwiIiwwLHsic3R5bGUiOnsiYm9keSI6eyJuYW1lIjoiZGFzaGVkIn19fV0sWzEsMywiIiwwLHsic3R5bGUiOnsiYm9keSI6eyJuYW1lIjoiZGFzaGVkIn19fV0sWzIsNl0sWzIsNV0sWzMsNiwiIiwwLHsibGV2ZWwiOjJ9XSxbMyw3XSxbNCw3XSxbNCw4XSxbNiw5LCIiLDAseyJzdHlsZSI6eyJib2R5Ijp7Im5hbWUiOiJkYXNoZWQifX19XSxbOSwxMCwiIiwwLHsic3R5bGUiOnsiYm9keSI6eyJuYW1lIjoiZGFzaGVkIn19fV1d
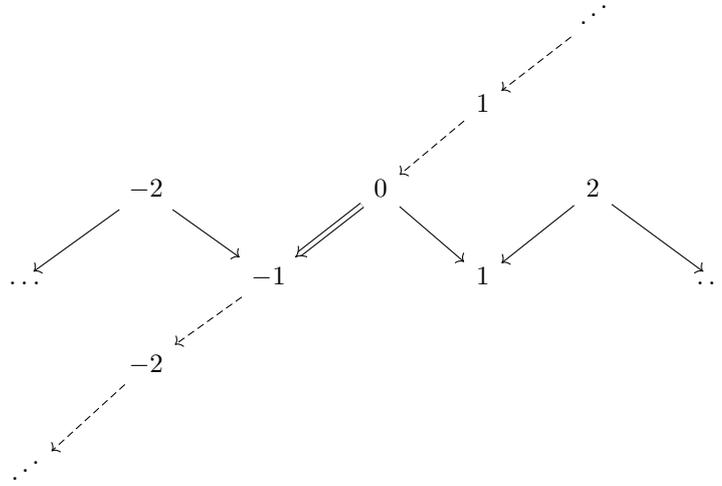
\begin{tikzcd}
	&&&&& \dots \\
	&&&& 1 \\
	& {-2} && 0 && 2 \\
	\ldots && {-1} && 1 & {} & \ldots \\
	& {-2} \\
	\dots
	\arrow[dashed, from=1-6, to=2-5]
	\arrow[dashed, from=2-5, to=3-4]
	\arrow[from=3-2, to=4-1]
	\arrow[from=3-2, to=4-3]
	\arrow[Rightarrow, from=3-4, to=4-3]
	\arrow[from=3-4, to=4-5]
	\arrow[from=3-6, to=4-5]
	\arrow[from=3-6, to=4-7]
	\arrow[dashed, from=4-3, to=5-2]
	\arrow[dashed, from=5-2, to=6-1]
\end{tikzcd}$$
Where the normal arrow represents the element of $\mathbf{Zig}$, the dashed arrows are the elements from $\mathbf{Z}$ and the double arrow represent the elements ``glued together'' under the pushout.
\end{example}

Notice that as $\emptyset$ is an initial object of $\textbf{\textup{prosets}}_{\textup{fin}}$ then this category has coequalizers [Proposition 5.14 \cite{Category_theory_101}]. We then get the following:

\begin{corollary}
The category $\textbf{\textup{prosets}}_{\textup{fin}}$ has all colimits.
\end{corollary}

\begin{proof}
By Lemmas \ref{lemma:existence coproducts posets} and \ref{lemma:existence of pushouts} and Proposition [Corollary 5.22, Theorem 5.23 \cite{Category_theory_101}] our results follows.
\end{proof}

And as the category of $\textbf{\textup{prosets}}_{\textup{fin}}$ has only FCC maps and locally finite ordered sets, we get the following:

\begin{corollary}
Colimits of FCC maps are FCC maps and they preserve the property of being locally finite.
\end{corollary}

In other words, direct limits, pushouts and disjoint unions of locally posets under FCC maps are locally finite posets. We now will use these results to show that coequalizers are mapped to equalizers via $\mathbf{M}_{\_}(P)$.

\begin{remark}
Let $f, g:\Lambda^1\rightarrow \Lambda^2$ be two maps. Let $\{\Lambda_i\}_{i\in I}$ the set of components of $\Lambda^1$ and $\{\Lambda_j\}_{j\in J}$ the set of components of $\Lambda^2$. We define $\textup{Coeq}(f, g):=\Lambda^2/\equiv_{\textup{Coeq}}$ as follows:
\begin{enumerate}
    \item For all $t\in \Lambda^1$, $f(t)\equiv_{\textup{Coeq}}g(t)$.
    \item If $j\in J$ is such that there is $i\in I$, $t\in \Lambda_i$ with $f(t), g(t)\in \Lambda_j$ and $f(t)\neq g(t)$, then for every $s_1, s_2\in \Lambda_j$, $s_1\equiv_{\textup{Coeq}}s_2$.
\end{enumerate}
and the map $p:\Lambda^2\rightarrow \textup{Coeq}(f, g)$ as the quotient map. Notice that the first condition makes $p\circ f=p \circ g$, and the second condition is a minimal condition for the map $p$ to be FCC.
\end{remark}

\begin{lemma}\label{lemma:functor maps coequalizers to equalizers}
Let $f_1, f_2:\Lambda^1\rightarrow \Lambda^2$ and $p:\Lambda^2\rightarrow \textup{Coeq}(f_1, f_2)$ be the coequalizer of $f_1$ and $f_2$. Then $\mathbf{M}[p]$ is the equalizer of $\mathbf{M}[f_1], \mathbf{M}[f_2]$. 
\end{lemma}

\begin{proof}
First, let $f_1,f_2, p$ as in the statement. We first show that $\mathbf{M}[f_1]\circ\mathbf{M}[p]=\mathbf{M}[f_2]\circ\mathbf{M}[p]$. We then prove $\mathbf{M}[p]$ is the equalizer of $\mathbf{M}[f_1],\mathbf{M}[f_2]$.\\
  Notice that for $n=1, 2$:
$$\begin{array}{llll}
    \mathbf{M}[f_n]\circ\mathbf{M}[p]:& \textup{M}_{\textup{Coeq}(f_1, f_2)}(P) &\longrightarrow &  \textup{M}_{\Lambda^1}(P) \\
     & r^{\textup{Coeq}(f_1, f_2)} & \longmapsto & r^{\Lambda^1}\\
     & 1^{\{s\}} & \longmapsto & \left\{\begin{array}{ll}
                                            1^{(p\circ f_n)^{-1}(s)}   & \mbox{if }(p\circ f_n)^{-1}\neq\emptyset \\
                                            0     & \mbox{otherwise} 
     \end{array}\right. \\
     & e^{(s_1,s_2)}&\longmapsto & \left\{\begin{array}{ll}
                                          \sum_{t_1,t_2}e^{(t_1,t_2)}    &  \mbox{for }t_1\precnapprox t_2 \mbox{ and }p\circ f_n(t_1)=s_1, \\
                                           & p\circ f_n(t_2)=s_2\\
                                       0      &  \mbox{otherwise}
                                    \end{array}\right.    
\end{array}$$
Because $p$ is the coequalizer of $f_1$ and $f_2$ it follows that $\mathbf{M}[f_1]\circ\mathbf{M}[p]=\mathbf{M}[f_2]\circ\mathbf{M}[p]$. Notice that as $p$ is surjective, Remark \ref{remark:surj mapped to inj} implies $\mathbf{M}[p]$ is injective.\\
  Now let $\Lambda'$ be a poset and $h':{M}_{\Lambda'}(P)\rightarrow \textup{M}_{\Lambda^2}(P)$ be such that $\mathbf{M}[f_1]\circ h'=\mathbf{M}[f_2]\circ h'$ and $h'(r^{\Lambda'})=r^{\Lambda^2}$, for every $r\in P$. Define $h:\textup{M}_{\Lambda'}(P)\rightarrow \textup{M}_{\textup{Coeq}(f_1, f_2)}(P)$ as:
$$h(A)=\left\{\begin{array}{ll}
    (\mathbf{M}[p])^{-1}(h'(A)) & \mbox{if } (\mathbf{M}[p])^{-1}(h'(A))\neq\emptyset  \\
    0 & \mbox{otherwise.}
\end{array}\right.$$ 

\noindent This function is well-defined because $\mathbf{M}[p]$ is injective. Hence the following diagram commute:

$$% https://tikzcd.yichuanshen.de/#N4Igdg9gJgpgziAXAbVABwnAlgFyxMJZAFgBoAGAXVJADcBDAGwFcYkQBZAfWAB1eAMvQC2AIyj0uARgC+ACgAKAShAzS6TLnyEUAJgrU6TVu259BI8ZN3zlq9SAzY8BIuQM0GLNok49+QmISXADMtipqGs7abqS6hl4mvmYBlhIA5OGqhjBQAObwRKAAZgBOEMJIUjQ4EEjuIHAAFljFOEgAtCE0jPSiMIwKmi46IFhg2LAgnsY+IPzC9DhNosXAHDLIxZT2JeWViNUgtfU0za3thz19A0PRrr7jk2wz3uwLSytrG8h5O5EgMoVJDdY51RD6ECMcZzCTNXLTKE3QbDGKPCZYKavJIgJq7QH7EE1cFHXr9FH3UZPTEvIxvXxNdL4oEHSEnK5I8l3LQPMYYrF0nEfZardabND-SgyIA
\begin{tikzcd}
\textup{M}_{\textup{Coeq}(f_1, f_2)}(P) \arrow[rr, "{\mathbf{M}[p]}" description]                         &  & \textup{M}_{\Lambda^2}(P) \arrow[rr, "{\mathbf{M}[f_1]}" description, shift left=3] \arrow[rr, "{\mathbf{M}[f_2]}" description, shift right] &  & \textup{M}_{\Lambda^1}(P) \\
                                                                                   &  &                                                                                                                                 &  &                  \\
\textup{M}_{\Lambda'}(P) \arrow[uu, "h" description, dashed] \arrow[rruu, "h'" description] &  &                                                                                                                                 &  &                 
\end{tikzcd}$$
That is, $\mathbf{M}[p]$ is the equalizer of $\mathbf{M}[f_1], \mathbf{M}[f_2]$.
\end{proof} 

\begin{corollary}\label{coro:functor maps colimits to limits}
The incidence functor $\mathbf{M}_{\_}$ sends colimits to limits.
\end{corollary}

\begin{proof}
By Lemma \ref{lemma:functor maps coequalizers to equalizers} and Corollary \ref{coro:functor maps coproducts to products} we have that the functor $ ^{op}\circ\mathbf{M}_\_(P):\textbf{\textup{prosets}}_{\textup{fin}}\Rightarrow \textbf{TopRings}^{op}$ maps coequalizers to coequalizers and coproducts to coproducts. It then follows that it maps colimits to colimits [Exercise 2 \cite{categories_working_math}].
\end{proof}

\subsection{Generating the category \texorpdfstring{$\textbf{\textup{prosets}}_{\textup{fin}}$}{TEXT}}

In this section, we focus on showing that the prosets of the form $\mathbf{m}\leftarrow \mathbf{n}$, for $n, m\in\mathbb{N}$ can generate all locally finite prosets under direct limits, disjoint unions, and pushouts of FCC maps.\\
  The next result is the version of Theorem \ref{Prop: prodiscrete ring} in the category $\textbf{\textup{prosets}}_{\textup{fin}}$.

\begin{theorem}\label{thrm:direct limit posets finite intervals}
Let $\Lambda$ be an irreducible, locally finite proset. Then $\Lambda\cong\varinjlim_{\alpha\in\Gamma(\Lambda)}\alpha$, where $\Gamma(\Lambda)=\{\alpha\subset \Lambda: $ is finite and convex$\}$ with subset order.
\end{theorem}

\begin{proof}
For each $\alpha_1\subset\alpha_2\in \Gamma(\Lambda)$ let $j_{\alpha_2\rightarrow\alpha_1}:\alpha_1\rightarrow\alpha_2$ be the convex embedding given by the subset structure. It is clear that if $\alpha_1\subset\alpha_2\subset\alpha_3$ then $j_{\alpha_1\rightarrow\alpha_2}\circ j_{\alpha_2\rightarrow\alpha_3}=j_{\alpha_1\rightarrow\alpha_3}$, hence the direct limit $\varinjlim_{\alpha\in\Gamma(\Lambda)}\alpha$ exists. For each $\alpha_i\in\Gamma(\Lambda)$ let $i_{\alpha_i}:\alpha_i\rightarrow \Lambda$ be the embedding of $\alpha_i$ into $\Lambda$. Because the outer triangle of the following diagram commutes, there exists an unique $f:\varinjlim_{\alpha\in\Gamma(\Lambda)}\alpha\rightarrow \Lambda$ such that the whole diagram commutes.
$$% https://tikzcd.yichuanshen.de/#N4Igdg9gJgpgziAXAbVABwnAlgFyxMJZAJgBpiBdUkANwEMAbAVxiRAB136AnLMAKwZYAtgH1gnRmgAWdTn04BxOsOF0AFJwAyKgEZQ6ASgC+khjLohjpdJlz5CKAAykAzFVqMWbMxdHErGxAMbDwCIgAWNw96ZlZEDnYpWVEARkDbUIciMicYr3jEnWF9S2MPGCgAc3giUAAzbghhJBcQHAgkV2ohMEKDOGlKkB66XRgGAAU7MMcQPmxYEc84tnqMkEbmpDJ2zsRukAYxiems8ISFrCXqWO9L8V8U1OMNrZbEVOoOrtHxqZm2UuYEWrFuBTYWEeSXMKWIr2sDSaHy+e1af1OgIu8xB1zBK3uIH40OSdH8CKC7x2332bWO-zO9mxVxuBMKxIkML8LzeyOpaM+GIB5zmLPxd3ZJNhZOInF4VWkODo3CaAHcnmSeeVjEA
\begin{tikzcd}
                                                                                           &  & \Lambda                                                                         &  &                                                                                                                                                        \\
                                                                                           &  &                                                                                 &  &                                                                                                                                                        \\
                                                                                           &  & \varinjlim_{\alpha\in\Gamma(\Lambda)}\alpha \arrow[uu, "f" description, dashed] &  &                                                                                                                                                        \\
\alpha_2 \arrow[rruuu, "i_{\alpha_2}" description] \arrow[rru, "j_{\alpha_2}" description] &  &                                                                                 &  & \alpha_1 \arrow[lluuu, "i_{\alpha_1}" description] \arrow[llu, "j_{\alpha_1}" description] \arrow[llll, "j_{\alpha_2\rightarrow\alpha_1}" description]
\end{tikzcd}$$
By the direct limit property, $f$ is an FCC map and a direct limit of embeddings. Hence it is a convex embedding. To complete our proof, we will show that $f$ is surjective.\\
  Let $a\in \Lambda$. The set $\mathcal{N}_0(a)$ is a finite convex subset of $\Lambda$ containing $a$. It is clear that $f\circ j_{\mathcal{N}_0(a)}(a)=a$, hence $f$ is surjective.
\end{proof}

Lastly we can show a property which can be taken from the category $\textbf{prosets}_{\text{fin}}$ and translated to the context of matrix rings over $P$ indexed by a poset.

\begin{theorem}\label{thrm:creating prosets from countable class}
All objects in $\textbf{\textup{prosets}}_{\textup{fin}}$ are generated by $\mathbf{m}\leftarrow\mathbf{n}$, for $m, n\in\mathbb{N}$, and  under disjoint unions, direct limits and pushouts of convex embeddings.
\end{theorem}

\begin{proof}
We first prove by induction that we can generate all the irreducible, finite prosets. from The proof for this case will be by induction on the cardinality of $\Lambda$.\\
  Our generating set generates contains all irreducible posets of cardinalities $0$, $1$ and $2$, proving the base cases, as noted in Example \ref{ex:basic posets}. Assume the result is true for every irreducible $\Lambda$ with cardinality $k$, for some $k>2$.\\
  Let $\Lambda\notin\{\mathbf{i}\leftarrow\mathbf{j}\}_{i+j=k+1}$ be an irreducible proset with cardinality $k+1$. Let $a\neq b\in \Lambda$ such that  $\Lambda\backslash\mathcal{N}_0(a)$, $\Lambda\backslash\mathcal{N}_0(b)$, and $\Lambda\backslash(\mathcal{N}_0(a)\cup\mathcal{N}_0(b))$ are irreducible. Let $\alpha=\Lambda\backslash (\mathcal{N}_0(a)\cup\mathcal{N}_0(b))$, $\Lambda_1=\Lambda\backslash \mathcal{N}_0(a)$ and $\Lambda_2=\Lambda\backslash \mathcal{N}_0(b)$. Such elements $a,b$ exists as every finite connected graph with at least $3$ vertices has at least two non-cut vertices. The conditions on $\Lambda$ imply $\Lambda\backslash(\mathcal{N}_0(a)\cup\mathcal{N}_0(b))$ is a non-empty set.\\
  Define $i_1:\alpha\rightarrow \Lambda_1$, $i_2:\alpha\rightarrow \Lambda_2$ the embeddings given by the subset structure. It follows that $|\Lambda_1|,|\Lambda_2|,|\alpha|<k+1$ and $\Lambda$ is the pushout of these prosets, proving our induction.\\
  Theorem \ref{thrm:direct limit posets finite intervals} implies that all infinite irreducible prosets are generated by finite, irreducible prosets under direct limits. Proposition \ref{defi:disjoint union prosets} implies that if all the irreducible posets are generated, so are the reducible ones under disjoint unions.
\end{proof} 

\begin{corollary}\label{coro:building blocks for M_(P)}
The subclass of all rings of the form $\textup{M}_{\Lambda}(P)$, for $\Lambda$ a poset, is generated by $\textup{M}_{\mathbf{m}\leftarrow\mathbf{n}}(P)$, for $m,n\in\mathbb{N}$, under products, inverse limits and pullbacks.
\end{corollary}

\begin{proof}
A direct corollary from Corollary \ref{coro:functor maps colimits to limits} and Theorem \ref{thrm:creating prosets from countable class}.
\end{proof}

This corollary can be useful for proving general properties of these rings. Instead of proving for all of them individually, you can prove for the generators and show that these properties are preserved by inverse limits and pullbacks.

\section{Incidence functor and the group of units}\label{Incidence functor and the group of units}

\begin{definition}\index{$\textup{GL}_\Lambda(P)$|ndx}\index{general linear group indexed by a proset|ndx}
Given $P$ a ring and $\Lambda$ a proset. We define the \textbf{general linear group of $P$ in relation to $\Lambda$} as:
$$\textup{GL}_\Lambda(P):=\left\{ A\in \textup{M}_\Lambda(P)
:  A^{-1}\in \textup{M}_\Lambda(P)\right\}.$$
That is, the group of units of $\textup{M}_\Lambda(P)$.
\end{definition} 

\begin{definition}\index{$\textup{DET}(A)$|ndx}
Let $\Lambda$ be a locally finite proset and $P$ be a ring. Let $A\in \textup{M}_\Lambda(P)$. We define the set $\textup{DET}(A)$ as:
$$\textup{DET}(A)=\{\det(\pi_\alpha(A)):  \alpha\in \Gamma(\Lambda)\}\subset P^{\Gamma(\Lambda)}.$$
\end{definition}

The following is a directly corollary of Theorem \ref{thrm:invertible as product} and the fact the groups of the form $\textup{GL}_\Lambda(P)$ are inverse limits of finite matrix rings. 

\begin{proposition}\label{prop:conditions to have inverse on lambda}\index{absolute topological ring|ndx}
Let $P$ be a commutative ring and $\Lambda$ a preordered set, and $A\in \textup{M}_\Lambda(P)$. Then $A\in \textup{GL}_\Lambda(P)$ if, and only if, $\pi_\alpha(A)$ is invertible for all $\alpha\in\Gamma(\Lambda)$. That is, $A\in \textup{GL}_\Lambda(P)$ if, and only if, $\textup{DET}(A)\subset \prod_{\alpha\in\Gamma(\Lambda)}P^*$. 
\end{proposition}

A stronger version of Proposition \ref{prop:conditions to have inverse on lambda} is given at [Theorem 1.16 \cite{Matrices1973}].

\begin{corollary}\label{coro:matrices absolute topological}
Let $P$ be a commutative, absolute topological ring. Then for every finite preordered set $\Lambda$ we have that $\textup{M}_\Lambda(P)$ is an absolute topological ring.
\end{corollary}

\begin{proof}
Let $A$ be an element of $ \textup{GL}_\Lambda(P)$. By Theorem \ref{thrm:cayley-hamilton} we have that for every $\alpha\in\Gamma(\Lambda)$, the map $\pi_\alpha(A)\mapsto \textup{adj}(\pi_\alpha(A))$ is a polynomial, hence it is continuous. As $P$ is an absolute topological ring and $\det(\pi_{\alpha}(A))\in P^*$, we have by Theorem \ref{thrm:invertible as product} and Proposition \ref{prop:conditions to have inverse on lambda} that if $\pi_\alpha(A)\in \textup{GL}_\alpha(P)$, the map $\pi_\alpha(A)\mapsto \pi_\alpha(A^{-1})$ is continuous. As the inverse map is continous for all $\textup{GL}_\alpha(P)$, $\alpha\in\Gamma(\Lambda)$, the inverse limit property implies the map $A\mapsto A^{-1}$ is continuous for $\textup{GL}_\Lambda(P)$. 
\end{proof}

As $P$ is commutative, when $\Lambda$ is a poset we can define $\textup{GL}_\Lambda(P)$ as all the infinite matrices such that the diagonal elements belongs to $P^*$. In the case $P$ is not commutative this is not always the case, as shown in articles \cite{inverse_upper_that_is_lower} and \cite{inverse_upper_that_is_lower_2}.\\
  Given that the functor that sends a ring to its units group is left adjoint [Proposition 9.14 \cite{Category_theory_101}],  Corollaries \ref{coro:matrices absolute topological}, \ref{coro:functor maps colimits to limits}, and Theorem \ref{Prop: prodiscrete ring} allow us to prove the following for the case $\textup{GL}_\Lambda(P)$:

\begin{proposition}\label{prop:making group inverse limit}
Let $P$ be an absolute topological ring and $\Lambda$ a locally finite proset. The following are true:
\begin{enumerate}
    \item If $\Lambda$ is infinite, then $\textup{GL}_
\Lambda(P)\cong\varprojlim_{\alpha\in \Gamma(\Lambda)} \textup{GL}_\alpha(P)$ is a topological group.
    \item Given $\Lambda$ is a proset, then $\textup{GL}_\Lambda(P)$ can be generated by $\textup{GL}_{\mathbf{m}\leftarrow\mathbf{n}}(P)$ under products, inverse limits and pullbacks.
\end{enumerate} 
\end{proposition}

We can also describe some normal subgroups of $\textup{GL}_\Lambda(P)$ in a similar way to the ideals of $\textup{M}_\Lambda(P)$, as follows:
\ \textbf{Notation}:[Normal subgroups]\label{defi:normal subgroups}
Let $\Lambda$ a locally finite proset and $P$ a commutative ring. We denote the following normal subgroups:
\begin{itemize}
    \item Given $s_1  \preceq s_2$ in $\Lambda$ we define 
    $$ N^{\Lambda}_{[s_1,s_2]}(P)=\{A\in \textup{GL}_{\Lambda}(P):  \text{if }s\in [s_1, s_2], A_{s,s}=1 \text{, and if } t_1\neq t_2\in[s_1, s_2],  A_{t_1,t_2}= 0 \}.$$ 
    This normal subgroup has quotient 
    $$\textup{GL}_{\Lambda}(P)/N^{\Lambda}_{[s_1, s_2]}(P)\cong \textup{GL}_{[s_1, s_2]}(P).$$
    \item For a convex set $\Lambda'\subset \Lambda$, the normal subgroup
    $$N^{\Lambda}_{\Lambda'}=\bigcap_{s_1  \preceq s_2\in \Lambda'}N^{\Lambda}_{[s_1, s_2]}.$$
    This normal subgroup has quotient $$\textup{GL}_{\Lambda}(P)/N^{(\Lambda,  \preceq)}_{\Lambda'}\cong \textup{GL}_{\Lambda'}(P).$$
    \item For $\{\Lambda_i\}_{i\in I}$ a locally convex set of $\Lambda$ we define the normal subgroup 
    $$N^{\Lambda}_{\{\Lambda_i\}_{i\in I}}=\bigcap_{i\in I}N^{\Lambda}_{\Lambda_i}.$$
    This normal subgroup has quotient 
    $$\textup{GL}_{\Lambda}(P)/N^{\Lambda}_{\{\Lambda_i\}_{i\in I}}\cong \prod_{i\in I} \textup{GL}_{\Lambda_i}(P).$$
\end{itemize}

\begin{example}
Let $\Lambda=\mathbf{3}<$, $P=\mathbb{Z}$ and $\alpha=\{0,1\}$. Then $N^{\mathbf{3}<}_\alpha$ is the normal subgroup with elements of the form:
$$\left(\begin{array}{lll}
1 & 0 & a_{0,2} \\
0 & 1 & a_{1,2} \\
0 & 0 & a_{2,2} 
\end{array}\right)$$
where $a_{2,2}\in \{-1,1\}$. 
\end{example}

The following is the equivalent for Corollary \ref{coro:open iff matrices} for the groups $\textup{GL}_\Lambda(P)$.

\begin{corollary}\label{coro:open iff matrices group}
Let $P$ a commutative, absolute topological ring, $\Lambda$ a locally finite proset and $\{\Lambda_i\}_{i\in I}\subset \Lambda$ locally convex. Then the following are true:
\begin{enumerate}
    \item The normal subgroup $N^{\Lambda}_{\{\Lambda_i\}_{i\in I}}$ is closed in $\textup{GL}_\Lambda(P)$.
    \item The normal subgroup $N^{\Lambda}_{\{\Lambda_i\}_{i\in I}}$ is open if, and only if, $\bigsqcup_{i\in I} \Lambda_i$ is a finite set and $P$ has the discrete topology.
\end{enumerate}
\end{corollary}

One can also get an equivalent to Theorem \ref{thrm:functor construction} for the groups of units. For that we denote $\mathbf{U}:\mathbf{Rings}\rightarrow \mathbf{Grps}$ the functor that maps a ring to its group of units. The following is a direct consequence from the fact right adjoint functors preserve limits [Proposition 9.14 \cite{Category_theory_101}], Theorem \ref{thrm:functor construction}, Corollary \ref{coro:functor maps colimits to limits} and Corollary \ref{coro:matrices absolute topological}:

\begin{theorem}\index{contravariant functor|ndx}
Let $P$ be a commutative, absolute topological ring. The functor $\mathbf{GL}_{\_}(P):=\mathbf{U}\circ \mathbf{M}_{\_}(P)$ is a contravariant functor from the category $\mathbf{prosets_{fin}}$ to the category  of topological groups, $\mathbf{TopGrps}$. This functor maps colimits to limits.
\end{theorem}

We can use these results to classify when the group of units of an incidence ring is solvable, prosolvable, or neither. Before stating and giving an alternative proof for \cite[Lemma 3.23]{silva2024topologicallysimpleinfinitematrix}, we will prove the following:

\begin{lemma}\label{lemma:seeing as upper triangular}
    Let $\Lambda$ be a finite poset and $P$ ring. Then there exists a embedding 
$$i_{\Lambda,n}:\textup{M}_{\Lambda}(P)\rightarrow \textup{M}_{<n}(P)$$
    where $n=|\Lambda|$.
\end{lemma}

\begin{proof}
    Let $\Lambda=\{s_0,s_1,\ldots, s_{n-1}\}$. Let $m_1\subset \Lambda$ be the set of all minimal elements. Then for every $t_1\neq t_2\in m_1$, $t_1\nleq t_2$ and $t_2\nleq t_1$, as otherwise we would have a contradiction in relation to the minimality of the elements of $m_1$. Inductively, define $$m_{k+1}\subset \Lambda\backslash \bigcup_{i<k+1}m_i$$
    as the set of all minimal elements of $\Lambda\backslash \cup_{i<k+1}m_i$. As the poset $\Lambda$ is finite, there is some $N\in\mathbb{N}$ such that $m_N\neq \emptyset$ and $m_{N+1}=\emptyset$.\\
    \indent Using the sets described above, we can order the elements of $\Lambda$ so that if $s_i\in m_j$ then $s_{i+1}\in m_j\cup m_{j+1}$. It then follows that for $0\leqslant i<j<n$, if $s_i\in m_k$ we have that either $s_i<s_j$ or $s_i\ngeq s_j$, as if $s_i\geq s_j$ we would have a contradiction to the minimality of $s_i$  in $\Lambda\backslash \cup_{l<k} m_l$. We thus have the following presentation for the elements of $\textup{M}_{\Lambda}(P)$
     $$\left(\begin{array}{lllll}
a_{s_0,s_0} & a_{s_0,s_1} &  a_{s_0,s_2} & \ldots & a_{s_0,s_{n-1}} \\
0 & a_{s_1,s_1} & a_{s_1,s_2} & \ldots & a_{s_1,s_{n-1}} \\
0 & 0 & a_{s_2,s_2} & \ldots & a_{s_2,s_{n-1}} \\
\vdots & \vdots & \vdots & \ddots & \vdots \\
0 & 0 & 0 & \ldots & a_{s_{n-1},s_{n-1}} 
\end{array}\right)$$
    which induces an embedding from $\textup{M}_{\Lambda}(P)$ to $\textup{M}_{<n}(P)$ given by
\begin{align*}
i_{\Lambda,n}\colon \textup{M}_{\Lambda}(P) & \longrightarrow\textup{M}_{<n}(P)\\
(a_{s_i,s_j})_{s_i,s_j\in \Lambda}&\longmapsto (a_{i,j})_{i,j<n},
\end{align*}
where $a_{i,j}=a_{s_i,s_j}$ for all $i,j<n$.
\end{proof}

In other words, incidence algebras over finite posets can always be seen as a ring of upper triangular matrices. \\
\indent For the next lemma we will use the notation $G^{(n)}=[G^{ n-1},G^{n-1}]$ to describe the $n$-th element at the derived series of a group.

\begin{lemma}\cite[Lemma 3.23]{silva2024topologicallysimpleinfinitematrix}
Assume $\Lambda$ is a partially ordered set and $P$ is a commutative, absolute topological ring. Then the group $\textup{GL}_\Lambda(P)$ is 
residually solvable. If $\Lambda$ is $n$-bounded then $\textup{GL}_\Lambda(P)$ is solvable and $\textup{GL}_\Lambda(P)^{(n)}=\{1\}$.    
\end{lemma}

\begin{proof}
    For every $\alpha\in \Gamma(\Lambda)$, that is, $\alpha$ a finite convex subset of $\Lambda$ with $|\alpha|=n$, let $i_{\alpha,n}$ be the embedding defined at Lemma \ref{lemma:seeing as upper triangular}. Notice that, by restricting the map to the commutator subgroups, we have a map
    $$i^{(1)}_{\alpha,n}:\textup{GL}_{\alpha}^{(1)}(P)\rightarrow \textup{GL}_{<n}^{(1)}(P),$$
    which is also an embedding. Inductively we can define the embeddings $i^{(j)}_{\alpha,n}:\textup{GL}_{\alpha}^{(j)}(P)\rightarrow \textup{GL}_{<n}^{(j)}(P)$. Notice that since $\textup{GL}_{<n}^{(n)}(P)=\{1\}$ and $i^{(n)}_{\alpha,n}$ is an embedding, we have that $\textup{GL}^{(n)}_{\alpha}(P)=\{1\}$, that is, $\textup{GL}_{\alpha}(P)$ is solvable. As this is the case for all $\alpha\in \Gamma(\Lambda)$ we have, by Proposition \ref{prop:making group inverse limit}, that $\textup{GL}_{\Lambda}(P)$ is an inverse limit of solvable groups, hence it is prosolvable.\\
    \indent If $\Lambda$ is $n$-bounded, then for every $\alpha\in \Gamma(\Lambda)$ we have that $\textup{GL}^{(n)}_\alpha(P)=\{1\}$. As the product of solvable groups is solvable, we have that 
    $$G:=\prod_{\alpha\in \Gamma(\Lambda)}\textup{GL}_{\alpha}(P)$$
    is also solvable, and $G^{(n)}=\{1\}$. As $\textup{GL}_{\Lambda}(P)$ is a subgroup of the product defined above, we have that $\textup{GL}^{(n)}_{\Lambda}(P)$, as desired.
\end{proof}

\begin{proposition}
    Assume that $\Lambda$ is a preordered set that is not a poset, and $\mathbb{F}$ is a field with at least four elements. The group $\textup{GL}_{\Lambda}(\mathbb{F})$ is not solvable nor prosolvable.
\end{proposition}

\begin{proof}
    As $\Lambda$ is a preordered set, there is $s\in \Lambda$ such that $[s,s]=\{z\in \Lambda :\ s\preceq z \preceq s\}\neq \{s\}$. Fix one $s$ satisfying this property and let $\alpha=[s,s]$. Notice that the group $\textup{GL}_{\alpha}(\mathbb{F})$ is isomorphic to $\textup{GL}_n(\mathbb{F})$, the group of $n\times n$ invertible square matrices over $\mathbb{F}$, where $n=|\alpha|$. As $[s,s]\neq \{1\}$ then $n>1$, hence the group $\textup{GL}_{\alpha}(\mathbb{F})$ is not solvable nor prosolvable, as $[\textup{GL}_{n}(\mathbb{F}),\textup{GL}_{n}(\mathbb{F})]=\textup{SL}_{n}(\mathbb{F})$ is a perfect group. \\
    \indent Let 
    $$\textup{G}=\{A\in \textup{GL}_{\Lambda}(\mathbb{F}):\ \text{ if }s_1\notin \alpha \text{ and }s_1\neq s_2 \text{ then } A_{s_1,s_1}=1,\ A_{s_1,s_2}=0 \text{ and }A_{s_2,s_1}=0 \}.$$
    This is a closed subgroup of $\textup{GL}_\Lambda(\mathbb{F})$ that is isomorphic to $\textup{GL}_{\alpha}(\mathbb{F})\cong \textup{GL}_{n}(\mathbb{F})$, as it only allows the coordinates indexed by $\alpha$ to be non-trivial. As $\textup{GL}_{\Lambda}(\mathbb{F})$ has a subgroup that is not solvable nor prosolvable, the result follows.
\end{proof}

\noindent \small{\textbf{Acknowledgements}  This article is part of my PhD thesis written under the supervision of George Willis, Colin Reid, and Stephan Tornier. I would like to thank them for advising me while writing the article and proofreading it. A special thanks as well to Mykola Khrypchenko for giving me references on the past work done on incidence rings, and for some nice discussion on some of the properties of such rings.\\
\textbf{Funding} The author received a PhD scholarship coming from Australian Research Council, under the Australian Laureate Fellowship ARC ``Zero-dimensional symmetry and its ramifications'' (FL170100032) for Professor George Willis. The work was entirely done during the author PhD.
%\textbf{Author Contributions} João Vitor Pinto e Silva is the author of this paper.\\
}

\bibliographystyle{elsarticle-harv} 
\bibliography{main.bib}

\begin{thebibliography}{16}
\expandafter\ifx\csname natexlab\endcsname\relax\def\natexlab#1{#1}\fi
\providecommand{\url}[1]{\texttt{#1}}
\providecommand{\href}[2]{#2}
\providecommand{\path}[1]{#1}
\providecommand{\DOIprefix}{doi:}
\providecommand{\ArXivprefix}{arXiv:}
\providecommand{\URLprefix}{URL: }
\providecommand{\Pubmedprefix}{pmid:}
\providecommand{\doi}[1]{\href{http://dx.doi.org/#1}{\path{#1}}}
\providecommand{\Pubmed}[1]{\href{pmid:#1}{\path{#1}}}
\providecommand{\bibinfo}[2]{#2}
\ifx\xfnm\relax \def\xfnm[#1]{\unskip,\space#1}\fi
%Type = Article
\bibitem[{Abrams et~al.(1999)Abrams, Haefner and del Rio}]{Matrices1999}
\bibinfo{author}{Abrams, G.}, \bibinfo{author}{Haefner, J.}, \bibinfo{author}{del Rio, A.}, \bibinfo{year}{1999}.
\newblock \bibinfo{title}{The isomorphism problem for incidence rings}.
\newblock \bibinfo{journal}{Pacific Journal of Mathematics} \bibinfo{volume}{187}, \bibinfo{pages}{201--214}.
%Type = Article
\bibitem[{Abrams et~al.(2002)Abrams, Haefner and del Rio}]{Matrices2002}
\bibinfo{author}{Abrams, G.}, \bibinfo{author}{Haefner, J.}, \bibinfo{author}{del Rio, A.}, \bibinfo{year}{2002}.
\newblock \bibinfo{title}{Corrections and addenda to 'the isomorphism problem for incidence rings'}.
\newblock \bibinfo{journal}{Pacific Journal of Mathematics} \bibinfo{volume}{207}, \bibinfo{pages}{497--506}.
%Type = Article
\bibitem[{Asplund(1959)}]{inverse_upper_that_is_lower_2}
\bibinfo{author}{Asplund, E.}, \bibinfo{year}{1959}.
\newblock \bibinfo{title}{Inverses of matrices $\{a_{ij}\}$ which satisfy $a_{ij} = 0$ for $j>i+p$.}
\newblock \bibinfo{journal}{MATHEMATICA SCANDINAVICA} \bibinfo{volume}{7}, \bibinfo{pages}{57–60}.
\newblock \URLprefix \url{https://www.mscand.dk/article/view/10561}, \DOIprefix\doi{10.7146/math.scand.a-10561}.
%Type = Book
\bibitem[{Awodey(2010)}]{Category_theory_101}
\bibinfo{author}{Awodey, S.}, \bibinfo{year}{2010}.
\newblock \bibinfo{title}{Category Theory}.
\newblock Oxford Logic Guides 52. \bibinfo{edition}{2nd} ed., \bibinfo{publisher}{Oxford University Press}, \bibinfo{address}{Oxford}.
%Type = Article
\bibitem[{Belding(1973)}]{Matrices1973}
\bibinfo{author}{Belding, W.R.}, \bibinfo{year}{1973}.
\newblock \bibinfo{title}{{Incidence rings of pre-ordered sets.}}
\newblock \bibinfo{journal}{Notre Dame Journal of Formal Logic} \bibinfo{volume}{14}, \bibinfo{pages}{481 -- 509}.
\newblock \URLprefix \url{https://doi.org/10.1305/ndjfl/1093891102}, \DOIprefix\doi{10.1305/ndjfl/1093891102}.
%Type = Article
\bibitem[{Bernkopf(1968)}]{history_infinite_matrices}
\bibinfo{author}{Bernkopf, M.}, \bibinfo{year}{1968}.
\newblock \bibinfo{title}{A history of infinite matrices: A study of denumerably infinite linear systems as the first step in the history of operators defined on function spaces}.
\newblock \bibinfo{journal}{Archive for History of Exact Sciences} \bibinfo{volume}{4}, \bibinfo{pages}{308--358}.
\newblock \URLprefix \url{http://www.jstor.org/stable/41133274}.
%Type = Article
\bibitem[{Dǎscǎlescu and van Wyk(1996)}]{Matrices1996}
\bibinfo{author}{Dǎscǎlescu, S.}, \bibinfo{author}{van Wyk, L.}, \bibinfo{year}{1996}.
\newblock \bibinfo{title}{Do isomorphic structural matrix rings have isomorphic graphs?}
\newblock \bibinfo{journal}{Proceedings of the American Mathematical Society} \bibinfo{volume}{124}, \bibinfo{pages}{1385--1391}.
\newblock \URLprefix \url{http://www.jstor.org/stable/2161446}.
%Type = Article
\bibitem[{Froelich(1985)}]{Matrices1985}
\bibinfo{author}{Froelich, J.}, \bibinfo{year}{1985}.
\newblock \bibinfo{title}{The isomorphisms problem for incidence rings}.
\newblock \bibinfo{journal}{Illinois Journal of Mathematics} \bibinfo{volume}{29}, \bibinfo{pages}{142--152}.
%Type = Article
\bibitem[{Haack(1984)}]{Matrices1984}
\bibinfo{author}{Haack, J.K.}, \bibinfo{year}{1984}.
\newblock \bibinfo{title}{Isomorphisms of incidence rings}.
\newblock \bibinfo{journal}{Illinois Journal of Mathematics} \bibinfo{volume}{28}, \bibinfo{pages}{676--683}.
%Type = Article
\bibitem[{Ho{\l{}}ubowski(2002)}]{inverse_upper_that_is_lower}
\bibinfo{author}{Ho{\l{}}ubowski, W.}, \bibinfo{year}{2002}.
\newblock \bibinfo{title}{An inverse matrix of an upper triangular matrix can be lower triangular}.
\newblock \bibinfo{journal}{Discussiones Mathematicae. General Algebra and Applications} \bibinfo{volume}{22}.
\newblock \DOIprefix\doi{10.7151/dmgaa.1055}.
%Type = Article
\bibitem[{Khrypchenko(2010)}]{Matrices2010}
\bibinfo{author}{Khrypchenko, M.}, \bibinfo{year}{2010}.
\newblock \bibinfo{title}{Finitary incidence algebras of quasiorders}.
\newblock \bibinfo{journal}{Matematychni Studiï} \bibinfo{volume}{34}.
%Type = Book
\bibitem[{Lane(1971)}]{categories_working_math}
\bibinfo{author}{Lane, S.M.}, \bibinfo{year}{1971}.
\newblock \bibinfo{title}{Categories for the working mathematician}.
\newblock \bibinfo{edition}{6} ed., \bibinfo{publisher}{SPRINGER-VERLAG}, \bibinfo{address}{New York}.
%Type = Article
\bibitem[{Parmenter et~al.(1990)Parmenter, Schmerl and Spiegel}]{Matrices1990}
\bibinfo{author}{Parmenter, M.}, \bibinfo{author}{Schmerl, J.}, \bibinfo{author}{Spiegel, E.}, \bibinfo{year}{1990}.
\newblock \bibinfo{title}{Isomorphic incidence algebras}.
\newblock \bibinfo{journal}{Advances in Mathematics} \bibinfo{volume}{84}, \bibinfo{pages}{226--236}.
\newblock \URLprefix \url{https://www.sciencedirect.com/science/article/pii/000187089090046P}, \DOIprefix\doi{https://doi.org/10.1016/0001-8708(90)90046-P}.
%Type = Misc
\bibitem[{e~Silva(2024)}]{silva2024topologicallysimpleinfinitematrix}
\bibinfo{author}{e~Silva, J.V.P.}, \bibinfo{year}{2024}.
\newblock \bibinfo{title}{Topologically simple infinite matrix groups indexed by ordered sets}.
\newblock \URLprefix \url{https://arxiv.org/abs/2410.23316}, \href{http://arxiv.org/abs/2410.23316}{{\tt arXiv:2410.23316}}.
%Type = Article
\bibitem[{Stanley(1970)}]{Matrices1970}
\bibinfo{author}{Stanley, R.P.}, \bibinfo{year}{1970}.
\newblock \bibinfo{title}{Structure of incidence algebras and their automorphism groups}.
\newblock \bibinfo{journal}{Bulletin of the American Mathematical Society} \bibinfo{volume}{76}, \bibinfo{pages}{1236--1239}.
\newblock \URLprefix \url{https://api.semanticscholar.org/CorpusID:2366603}.
%Type = Article
\bibitem[{Voss(1980)}]{Matrices1980}
\bibinfo{author}{Voss, E.R.}, \bibinfo{year}{1980}.
\newblock \bibinfo{title}{On the isomorphism problem for incidence rings}.
\newblock \bibinfo{journal}{Illinois Journal of Mathematics} \bibinfo{volume}{24}, \bibinfo{pages}{624--638}.
\newblock \URLprefix \url{https://api.semanticscholar.org/CorpusID:122444257}.

\end{thebibliography}

\end{document}